\documentclass[11pt]{article}

\usepackage{algorithm}
\usepackage{algorithmic}
\usepackage{amsmath, amssymb, amsthm}
\usepackage{enumitem}
\usepackage{fullpage}
\usepackage{times}
\usepackage{natbib}
\usepackage{graphicx}
\usepackage{url}
\usepackage{hyperref} 

\newcommand{\ip}[2]{\left\langle #1, #2 \right\rangle}
\newcommand{\abs}[1]{\left| #1 \right|}
\newcommand{\grad}{\nabla}

\newcommand{\ba}{\boldsymbol{a}}
\newcommand{\bb}{\boldsymbol{b}}

\newcommand{\bell}{\boldsymbol{\ell}}

\newcommand{\bg}{\boldsymbol{g}}
\newcommand{\bx}{\boldsymbol{x}}
\newcommand{\bu}{\boldsymbol{u}}
\newcommand{\by}{\boldsymbol{y}}

\newcommand{\bz}{\boldsymbol{z}}

\newcommand{\bv}{\boldsymbol{v}}

\newcommand{\bbeta}{\boldsymbol{\beta}}

\newcommand{\argmin}{\mathop{\mathrm{argmin}}}
\newcommand{\argmax}{\mathop{\mathrm{argmax}}}

\newcommand{\field}[1]{\mathbb{#1}}

\newcommand{\R}{\field{R}}

\newcommand{\btheta}{\boldsymbol{\theta}}

\newcommand{\norm}[1]{\left\|{#1}\right\|}

\newcommand{\sign}{{\rm sign}}

\DeclareMathOperator*{\Exp}{\mathbf{E}}

\newtheorem{theorem}{Theorem}
\newtheorem{lemma}[theorem]{Lemma}

\newtheorem{definition}[theorem]{Definition}

\title{Parameter-free Stochastic Optimization of Variationally Coherent Functions}

\author{
 Francesco Orabona\\
 Boston University, Boston, MA\\
 \url{francesco@orabona.edu}
 \and
 D\'avid P\'al\\
 New York, NY\\
 \url{davidko.pal@gmail.com}\\
}

\begin{document}

\maketitle

\begin{abstract}%
We design and analyze an algorithm for first-order stochastic optimization of a large class of functions on $\R^d$. In particular, we consider the \emph{variationally coherent} functions which can be convex or non-convex. The iterates of our algorithm on variationally coherent functions converge almost surely to the global minimizer $\bx^*$.
Additionally, the very same algorithm with the same hyperparameters, after $T$ iterations guarantees on convex functions that the expected suboptimality gap is bounded by $\widetilde{O}(\norm{\bx^* - \bx_0} T^{-1/2+\epsilon})$ for any $\epsilon>0$. It is the first algorithm to achieve both these properties at the same time.
Also, the rate for convex functions essentially matches the performance of parameter-free algorithms.
Our algorithm is an instance of the Follow The Regularized Leader algorithm with the added twist of using \emph{rescaled gradients} and time-varying linearithmic regularizers.
\end{abstract}

\section{Introduction}

We consider the problem of finding the minimizer of a differentiable function
$F:\R^d \to \R$ using access only to noisy gradients of the function. This is a
fundamental problem in stochastic optimization and machine learning. Indeed, a
plethora of algorithms have been proposed to solve this problem, some of them
being optimal with respect to some measure \citep[see, e.g.,][]{BottouCN16}.
However, the choice of an algorithm crucially depends on the assumptions on the
function $F$.

In an effort to go beyond convex functions, we focus on \emph{variationally
coherent} functions. Variationally coherent
functions~\citep{ZhouMBBG17,ZhouMBBG20} are defined by the property that, at any
point $\bx$, the vector pointing towards the optimal solution $\bx^*$ and the
negative gradient form an angle of at most 90 degrees. \citet{ZhouMBBG20} proved
that this class contains convex, quasi-convex,
$\tau$-star-convex~\citep{JoulaniGS17}, and pseudo-convex functions.

Smooth variationally coherent functions can be asymptotically minimized by
Stochastic Gradient Descent (SGD) with learning rates proportional to
$t^{-\alpha}$ where $t$ is the iteration number and $\alpha >
\tfrac{1}{2}$~\citep{ZhouMBBG20}. If the function happens to be also convex, SGD
with the same learning rate guarantees a convergence rate of $O
\left(\tfrac{1+\norm{\bx^* - \bx_0}^2}{T^{1-\alpha}} \right)$, where $\bx_0$ is
the initial point.

However, for convex functions a significantly better convergence rate
$O\left(\tfrac{1+\norm{\bx^* - \bx_0} \sqrt{\ln(1 + \norm{\bx^* -
\bx_0})}}{\sqrt{T}}\right)$ can be achieved by using the so-called
\emph{parameter-free}
algorithms~\citep[e.g.,][]{OrabonaP16,McMahanO14,CutkoskyO18} for online convex
optimization and averaging their iterates. Specifically, parameter-free
algorithms have better dependency on $\norm{\bx^* - \bx_0}$, that can be
arbitrarily large. Unfortunately, parameter-free algorithms are not known to
work for non-convex functions.

In this paper, we design a new parameter-free algorithm for convex functions
with bounded stochastic gradients that achieves convergence rate $O
\left(\tfrac{1+\norm{\bx^* - \bx_0} \ln (1 + \norm{\bx^* - \bx_0})}{T^{1-\alpha}}
\right)$ for $\alpha> \tfrac{1}{2}$ and \emph{at the same time} guarantees
asymptotic almost sure convergence to $\bx^*$ for variationally coherent
functions. No averaging of the iterates is required, we can guarantee
convergence directly for the last iterate. As far as we know, our algorithm is
the first of this kind.

Our algorithm is based on Follow The Regularized Leader (FTRL) algorithm with
time-varying linearithmic regularizer and \emph{rescaled gradients}. The
regularizer we use is similar to those used in other parameter-free algorithms.
Both FTRL and the regularizer are essential to guarantee the better dependency
of the convergence rate on $\norm{\bx^* - \bx_0}$ for convex functions. On the
other hand, rescaling of the gradients is needed to guarantee convergence for
variationally coherent functions and it is reminiscent of SGD. So our algorithm
can be viewed as a novel combination of FTRL and SGD and it might be of
independent interest.

The rest of the paper is organized as follows. In
Section~\ref{section:related-work}, we discuss related work. In
Section~\ref{section:def}, we formally define the problem and the class of
variationally coherent functions. The algorithm and the main results are stated
in Section~\ref{section:main-results}. In Section~\ref{section:ftrl}, we prove
basic properties of FTRL with rescaled gradients. In
Section~\ref{section:regularizer}, we present the time-varying regularizer and
its basic properties. Section~\ref{section:proofs} contains the proofs of the
main results. However, due to space limitations many supporting lemmas and their
proofs are deferred to appendices. Finally, in Section~\ref{section:conclusions}
we conclude the paper with discussion on limitations and future work.

\section{Related Work}
\label{section:related-work}

Follow The Regularized Leader (FTRL) was introduced as an algorithm for online
convex optimization (OCO) on linearized losses by \cite{ShalevS06,ShalevS07,Shalev-Shwartz07}. The name of
the algorithm comes from \citet{AbernethyHR08}. For offline
optimization, FTRL with linearized losses was introduced under the name Dual Averaging (DA) by\footnote{Note
that this paper by Nesterov is actually from 2005:
\url{https://papers.ssrn.com/sol3/papers.cfm?abstract_id=912637}}
\citet{Nesterov09}, with the main motivation of using non-decreasing weights for the (sub)gradients, contrary to Mirror Descent (MD)~\citep{NemirovskyY83}. Hence, even if the general scheme in \citet{Nesterov09} would support generic weights, DA in the stochastic setting is used with uniform weights. \citet{JuditskyKM20} propose another way to merge aspects of MD and DA. Also, they do not allow for generic time-varying regularizers that are essential here.

Parameter-free algorithms for OCO were introduced by
\citet{Orabona13,Orabona14}.  However, the seed of these ideas was already present in
\citet{StreeterM12,McMahanJ13}. The same ideas were developed in parallel for
the problem of the learning with expert
advice~\citep{ChaudhuriYH09,ChernovV10,LuoS15,KoolenVE15}. In fact, the name
\emph{parameter-free} originated in \citet{ChaudhuriYH09}. It is now clear that
these two approaches are fundamentally the same~\citep{OrabonaP16}. Our regularizers are inspired to the ones in \citet{KoolenVE15}, but with a max rather than a prior over $\beta$ (see \eqref{equation:psi-star-definition}) that gives simple closed forms.

Parameter-free algorithm can easily be used in the stochastic setting through
online-to-batch conversion~\citep{Cesa-bianchiCG02}. In the OCO setting, regret of parameter-free algorithms have optimal
dependency on $\norm{\bx^* - \bx_0}$, while regret of online gradient descent
has provably suboptimal dependency on $\norm{\bx^* - \bx_0}$
\citep{StreeterM12,CutkoskyB17}. It is not known if parameter-free algorithms
are optimal for stochastic optimization of convex Lipschitz functions, but it is
reasonable to assume that lower bound from \citet{StreeterM12} can be extended to
the stochastic optimization setting as well. As far as we know, there are no
stochastic optimization algorithms that achieves the convergence rate of
parameter-free algorithms for convex functions and guarantee convergence for
variationally coherent functions.

\citet{Zhang04b,ShamirZ13} proved convergence of the last iterate of SGD for
convex Lipschitz functions. However, the analysis critically relies on the
assumption of bounded domain. \citet{Orabona20} proved the convergence
of the last iterate of SGD on unbounded domains. We are not aware of other
proofs of convergence of last iterate of FTRL-based parameter-free algorithms
without changing the update rule \citep[e.g.,][]{Cutkosky19}.

Variationally coherent functions were introduced by~\cite{ZhouMBBG17,ZhouMBBG20}.
\citet{ZhouMBBG20} also points out the connection between variationally coherent functions and variational inequalities.
For this class of functions, \citet{ZhouMBBG17} proved almost
sure convergence for DA, without assuming a unique minimizer, but assuming Lipschitz gradients.
The classic analysis of \citet{Bottou98} used essentially the same definition and proved almost sure
convergence of SGD for smooth variationally coherent functions. Both definitions
can be traced back to the concept of \emph{pseudogradients} introduced by
\cite{PolyakT73}. They defined a pseudogradient of a function $F$ at a point
$\bx$ as any vector $\bg$ such that $\ip{\grad F(\bx)}{\Exp[\bg]} \ge 0$.
%Geometric interpretation of the condition is that the expected pseudogradient
%forms an acute angle with the gradient of the function.
They also introduced the
idea of having a pseudogradient of a surrogate objective function. In
particular, they considered the surrogate function
$\widetilde{F}(\bx)=\frac{1}{2} \norm{\bx^* - \bx}^2_2$, where $\bx^*$ is the
minimizer of $F$, so that the pseudogradient condition becomes $\ip{\Exp[\bg]}{\bx
-\bx^*} \ge 0$.

% The idea of having an algorithm that adapts to the underlying characteristic
% of the optimization problem without any change in its hyperparameters is a
% natural one. In the online convex optimization (OCO) community, people have
% designed algorithms that adapt, for example, to the sum of the squared norm of
% the gradients~\citep{McMahanS10,DuchiHS11}, scale~\citep{OrabonaP18}, strong
% convexity~\citep{vanErvenK16, CutkoskyO18}. In the offline optimization
% literature, this idea is sometimes called
% \emph{universality}~\citep{Nesterov15b}.

\section{Problem Setup and Notation}
\label{section:def}

\paragraph{Problem Setup}
We consider a model in which an algorithm has access to a stochastic first-order
oracle for a differentiable function $F:\R^d \to \R$. In each round
$t=1,2,\dots$, the algorithm computes an iterate $\bx_t \in \R^d$. The oracle
produces a stochastic gradient $\bg_t \in \R^d$ such that
\begin{equation}
\label{equation:stochastic-gradient}
\Exp \left[\bg_t ~\middle|~ \bx_1, \bx_2, \dots, \bx_t, \bg_1, \bg_2, \dots, \bg_{t-1} \right] = \grad F(\bx_t) \: .
\end{equation}
The iterate $\bx_t$ produced by the algorithm depends on the past gradients
$\bg_1, \bg_2, \dots, \bg_{t-1}$ and past iterates $\bx_1, \bx_2, \dots,
\bx_{t-1}$. Thus, even if the algorithm is deterministic, both $\bg_t$ and
$\bx_t$ are random variables. We denote by $\mathcal{F}_t$ the $\sigma$-algebra
generated by $\bx_1, \bx_2, \dots, \bx_t, \bg_1, \bg_2, \dots, \bg_{t-1}$. Using
this notation the condition \eqref{equation:stochastic-gradient} can be written
as
$$
\Exp \left[\bg_t ~\middle|~ \mathcal{F}_t \right] = \grad F(\bx_t) \: .
$$
The goal of the algorithm is to approach the minimizer
$\bx^*$ of $F$, that is,
\[
\lim_{t \to \infty} \bx_t = \bx^* \qquad \text{almost surely} \: .
\]
We make the additional assumption that
\begin{equation}
\label{equation:gradient-bound}
\norm{\bg_t} \le G \qquad \text{almost surely}
\end{equation}
where $G$ is a positive number. The assumption \eqref{equation:gradient-bound}
implies that the function $F$ is $G$-Lipschitz.

We design and analyze an algorithm for two classes of functions. The first class
is the class of differentiable convex functions with a minimizer $\bx^*$
(possibly not unique). The second class consists of variationally coherent
functions. Essentially the same class was studied already by~\citet{ZhouMBBG17}
and \citet{Bottou98}.
The class contains non-convex functions, see \citet{ZhouMBBG17} for examples.
\begin{definition}[Variatonally coherent function]
A function $F:\R^d \to \R$ is called \emph{variationally coherent} if
it is continuously differentiable, has a unique minimizer $\bx^*$, satisfies
\begin{equation}
\label{equation:variationally-coherent}
\ip{\grad F(\bx)}{\bx - \bx^*} \ge 0 \qquad \text{for all $\bx \in \R^d$}
\end{equation}
and the equality in \eqref{equation:variationally-coherent} holds if only if
$\bx = \bx^*$.
\end{definition}
We remark that \eqref{equation:variationally-coherent} is satisfied if $F$ is a
convex differentiable function with a minimizer $\bx^*$. Furthermore, any
function that is continuously differentiable and strictly convex with a unique
minimizer is necessarily variationally coherent.

\paragraph{Notation}
We denote by $\ip{\bx}{\by}$ the standard inner product on $\R^d$ and
$\norm{\bx} = \sqrt{\ip{\bx}{\bx}}$ is the Euclidean norm. \emph{Fenchel
conjugate} of a function $f:\R^d \to \R \cup \{-\infty, +\infty\}$ is the
function $f^*:\R^d \to \R \cup \{-\infty,+\infty\}$ defined as $f^*(\btheta) = \sup_{\bx \in \R^d} \ip{\btheta}{\bx} - f(\bx)$.
\emph{Bregman divergence} associated with a differentiable function $f:\R^d \to \R$ is
the function $B_f:\R^d \times \R^d \to \R$ defined as
$B_f(\bx, \by) = f(\bx) - f(\by) - \ip{\grad f(\by)}{\bx - \by}$.
Bregman divergence associated with a convex differentiable function is
non-negative. We will use this property throughout the paper.

\section{Main Results}
\label{section:main-results}

We propose the following algorithm for our setting. The algorithm is a Follow
The Regularized Leader (FTRL) algorithm with a particular sequence of
regularizers operating on the sequence of \emph{rescaled gradients} $\bell_t =
\eta_t \bg_t$, $t=1,2,\dots$. Section~\ref{section:ftrl} gives a detailed
explanation of FTRL. We call the scale factors $\eta_1, \eta_2, \dots$
\emph{learning rates}.

\begin{algorithm}[H]
\caption{\textsc{FTRL with rescaled gradients and linearithmic regularizer}}
\label{algorithm:ftrl-rescaled-gradients-exponetial-potential}
\begin{algorithmic}[1]
{
\REQUIRE{Initial point $\bx_0 \in \R^d$, sequence of learning rates $\eta_1, \eta_2, \dots$}
\STATE{Initialize $S_0^2=4, Q_0=0, \btheta_0 = \boldsymbol{0}$}
\FOR{$t=1,2,\dots$}
\STATE{Output
\[
\bx_t \leftarrow \bx_0+\begin{cases}
\dfrac{\norm{\btheta_{t-1}}}{2S_{t-1}^2}\exp\left( \dfrac{\norm{\btheta_{t-1}}^2}{4S_{t-1}^2} - Q_{t-1} \right) &   \text{if $\norm{\btheta} \le S_{t-1}^2$,} \\[0.5cm]
\dfrac{\btheta_{t-1}}{2\norm{\btheta_{t-1}}}\exp\left( \dfrac{\norm{\btheta_{t-1}}}{2} - \frac{1}{4} S_{t-1}^2 - Q_{t-1} \right) &  \text{if $\norm{\btheta} > S_{t-1}^2$.}
\end{cases}
\]}
\STATE{Receive stochastic gradient $\bg_t \in \R^d$ such that $\Exp \left[\bg_t ~\middle|~ \mathcal{F}_t \right] = \nabla F(\bx_t)$ and $\norm{\bg_t} \le G$}
\STATE{Compute rescaled gradient $\bell_t \leftarrow \eta_t \bg_t$}
\STATE{Update $S_t^2 \leftarrow S_{t-1}^2 + \norm{\bell_t}^2$}
\STATE{Update $Q_t \leftarrow Q_{t-1} + \frac{\norm{\bell_t}^2}{S_t}$}
\STATE{Update sum of negative rescaled gradients $\btheta_t \leftarrow \btheta_{t-1} - \bell_t$}
\ENDFOR
}
\end{algorithmic}
\end{algorithm}

We prove three results about the algorithm. Their proofs can be found in
Section~\ref{section:proofs} and the proofs of supporting lemmas in Appendices
\ref{section:variationally-coherent}, \ref{section:convex-case}, and
\ref{section:last-iterate}.
Theorem~\ref{theorem:variationally-coherent-convergence} states that, under
general assumptions on the sequence of learning rates, the iterates $\bx_1,
\bx_2, \dots$ converge to the minimizer $\bx^*$ provided that $F$ is
variationally coherent. Theorems~\ref{theorem:convex-average-non-adaptive} and
\ref{theorem:convex-last-iterate} are $\widetilde{O}(G T^{-1/2 + \epsilon}
\norm{\bx^*-\bx_0})$ bounds on the speed of convergence of the function values to the
optimal value $F(\bx^*)$ provided that $F$ is convex.
Theorem~\ref{theorem:convex-average-non-adaptive} applies to the running average
of the iterates. Theorem~\ref{theorem:convex-last-iterate} applies to iterates
directly. The last two theorems require a particular sequence of learning rates
$\eta_t = \frac{1}{G t^{\alpha}}$, where $\alpha \in (\frac{1}{2}, 1)$.

\begin{theorem}[$\bx_t$ converges to $\bx^*$]
\label{theorem:variationally-coherent-convergence}
Let $F:\R^d \to \R$ be a variationally coherent function with minimizer $\bx^*$.
Assume the stochastic gradients satisfy \eqref{equation:stochastic-gradient} and
\eqref{equation:gradient-bound}. Assume that learning rate $\eta_t$ is a
non-negative $\mathcal{F}_t$-measurable random variable, $t=1,2,\dots$, and
there exists a real number $\gamma > 0$ such that
\begin{align}
\label{equation:learning-rates-assumption-1}
\sum_{t=1}^\infty \eta_t^2 \norm{\bg_t}^2 & < \gamma  & \text{almost surely,} \\
\label{equation:learning-rates-assumption-2}
\sum_{t=1}^\infty \eta_t & = + \infty  & \text{almost surely,} \\
\label{equation:learning-rates-assumption-3}
\eta_t & \le \frac{1}{G} & \text{for all $t=1,2,\dots$ almost surely.}
\end{align}
Then, the sequence $\bx_1, \bx_2, \dots$ generated by
Algorithm~\ref{algorithm:ftrl-rescaled-gradients-exponetial-potential} satisfies
$$
\lim_{t \to \infty} \bx_t = \bx^* \qquad \text{almost surely.}
$$
\end{theorem}

The assumption that $\eta_t$ is $\mathcal{F}_t$-measurable means that $\eta_t$
is an arbitrary function of $\bg_1, \bg_2, \dots \bg_{t-1}$ and $\bx_1, \bx_2,
\dots \bx_t$. This way, $\Exp[\eta_t \bg_t ~|~ \mathcal{F}_t] = \eta_t
\Exp[\bg_t ~|~ \mathcal{F}_t] = \eta_t \grad F(\bx_t)$. Importantly, $\eta_t$
cannot depend on $\bg_t$, see also discussion in \citet{LiO19}. Assumptions
\eqref{equation:learning-rates-assumption-1} and
\eqref{equation:learning-rates-assumption-2} are essentially the same as the
assumptions used in the convergence results for stochastic gradient descent
algorithm~\citep{RobbinsM51}. Assumption \eqref{equation:learning-rates-assumption-1} means that
$\sum_{t=1}^\infty \eta_t^2 \norm{\bg_t}^2$ as a random variable is bounded.
Assumption \eqref{equation:learning-rates-assumption-3} ensures that
$\norm{\bell_t} \le 1$ which is important for the underlying FTRL algorithm.
Note that learning rate $\eta_t = \frac{1}{G t^{\alpha}}$, where $\alpha \in
(\frac{1}{2}, 1)$, satisfies all these assumptions.

\begin{theorem}[Convergence rate of running average for convex functions]
\label{theorem:convex-average-non-adaptive}
Let $F:\R^d \to \R$ be a convex differentiable function with a (possibly
non-unique) minimizer $\bx^*$. Let $\alpha \in (\frac{1}{2}, 1)$. Suppose the
stochastic gradients satisfy \eqref{equation:stochastic-gradient} and
\eqref{equation:gradient-bound}.
Algorithm~\ref{algorithm:ftrl-rescaled-gradients-exponetial-potential} with
learning rate  $\eta_t = \frac{1}{Gt^\alpha}$ satisfies for all $T \ge 1$
\[
\Exp\left[ F(\overline{\bx}_T) \right] 
\le F(\bx^*) + \frac{G}{T^{1-\alpha}} \left(1 + \sqrt{5 + \frac{1}{2\alpha - 1}} \norm{\bx^*-\bx_0} \left[ 2\ln(1 + 2 \norm{\bx^*-\bx_0}) + 9 \sqrt{5 + \frac{1}{2\alpha -1}} \right]\right) \: ,
\]
where $\overline{\bx}_T = \frac{1}{T} \sum_{t=1}^T \bx_t$ is the running average
of the iterates.
\end{theorem}

\begin{theorem}[Convergence rate of last iterate for convex functions]
\label{theorem:convex-last-iterate}
Let $F:\R^d \to \R$ be a convex differentiable function with a (possibly
non-unique) minimizer $\bx^*$. Let $\alpha \in (\frac{1}{2}, 1)$. Suppose the
stochastic gradients satisfy \eqref{equation:stochastic-gradient} and
\eqref{equation:gradient-bound}.
Algorithm~\ref{algorithm:ftrl-rescaled-gradients-exponetial-potential} with
learning rate  $\eta_t = \frac{1}{Gt^\alpha}$ satisfies for all $T \ge 1$
\begin{multline*}
\Exp[F(\bx_T)] - F(\bx^*) \le \frac{G}{T^{1-\alpha}} \left( 2 + \frac{1}{e(2\alpha - 1)} \right) \left( \exp(S) + 3\norm{\bx^*-\bx_0} \right.  \\
\left. + 6(S + 2) \left(2 + S \norm{\bx^*-\bx_0} \left[ 2\ln(1 + 2 \norm{\bx^*-\bx_0}) + 9 S \right] \right) \right) \: ,
\end{multline*}
where $S =  \sqrt{5 + \frac{1}{2\alpha - 1}}$ and $e = 2.718\dots$ is Euler's
constant.
\end{theorem}

Theorem~\ref{theorem:convex-average-non-adaptive} and
\ref{theorem:convex-last-iterate} have the same assumptions.
Theorem~\ref{theorem:convex-last-iterate} is obviously stronger, however, its
proof is based on the one of Theorem~\ref{theorem:convex-average-non-adaptive}. For that reason we include both theorems.

\section{FTRL with Rescaled Gradients}
\label{section:ftrl}

As we said, Algorithm~\ref{algorithm:ftrl-rescaled-gradients-exponetial-potential} is a
special case of the FTRL\footnote{We prefer to use the name FTRL over DA because FTRL is more general: DA is a special case of FTRL when the losses are linear.} algorithm with
\emph{rescaled gradients} stated as
Algorithm~\ref{algorithm:ftrl-rescaled-gradients} below. The algorithm differs
from the standard FTRL algorithm in that gradients $\bg_t$ are rescaled by the
learning rate $\eta_t$. In other words,
Algorithm~\ref{algorithm:ftrl-rescaled-gradients} can be viewed as the standard
FTRL algorithm operating on the sequence $\bell_t = \eta_t \bg_t$,
$t=1,2,\dots$. As usual, the algorithm is specified by a sequence of functions
$\phi_1, \phi_2, \dots$ called \emph{regularizers}. The regularizer $\phi_t$ and
the learning rate $\eta_t$ can depend on the previous gradients $\bg_1, \bg_2,
\dots, \bg_{t-1}$. This way $\eta_t$ and $\phi_t$ are $\mathcal{F}_t$-measurable
random elements.

\begin{algorithm}
\caption{\textsc{FTRL with rescaled gradients}}
\label{algorithm:ftrl-rescaled-gradients}
\begin{algorithmic}[1]
{
\REQUIRE{Initial point $\bx_0\in R^d$, regularizers $\phi_1, \phi_2, \dots$, and learning rates $\eta_1, \eta_2, \dots$}
\STATE{Initialize $\btheta_0 = \boldsymbol{0}$}
\FOR{$t=1,2,\dots$}
\STATE{Predict $\bx_t \leftarrow \bx_0+\argmin_{\bx \in \R^d} \ \phi_t(\bx) - \ip{\btheta_{t-1}}{\bx}$}
\STATE{Receive gradient $\bg_t \in \R^d$}
\STATE{Compute rescaled gradient $\bell_t \leftarrow \eta_t \bg_t$}
\STATE{Update sum of negative rescaled gradients $\btheta_t \leftarrow \btheta_{t-1} - \bell_t$}
\ENDFOR
}
\end{algorithmic}
\end{algorithm}

In Section~\ref{section:regularizer} we present the sequence of regularizers that
gives rise to
Algorithm~\ref{algorithm:ftrl-rescaled-gradients-exponetial-potential} and we
make the derivation of the formulas used in algorithm.

Other choices of regularizers and learning rates are also possible. Two simple
special cases are worth mentioning. The first special case is the choice $\eta_t
= 1$ for all $t$ which recovers standard FTRL algorithm. The second special case
is the choice $\phi_t(\bx) = \frac{1}{2} \norm{\bx}^2$ for all $t$ which
recovers the standard stochastic/online gradient descent algorithm. However, in
general, the algorithm is neither FTRL, nor stochastic gradient descent, not
even online mirror descent algorithm.

\textbf{Remark} In the following, for simplicity of notation, we assume
$\bx_0=\boldsymbol{0}\in \R^d$. It is easy to obtain the results for any other
choice of $\bx_0$ with a simple translation of the coordinate system.

The analysis of both algorithms relies on Lemma~\ref{lemma:regret-bound}. The
proof of Lemma~\ref{lemma:regret-bound} uses
Lemma~\ref{lemma:ftrl-equality}~\citep{Orabona19}, we report its proof in
Appendix~\ref{section:ftrl-proofs} for completeness. Both lemmas are expressed
in terms of the objective function that FTRL minimizes in step $t$,
\begin{align}
\label{equation:ftrl-objective}
H_t(\bx) & = \phi_t(\bx) - \ip{\btheta_{t-1}}{\bx} = \phi_t(\bx) + \sum_{i=1}^{t-1} \ip{\bell_i}{\bx} && \text{for $t=1,2,\dots$} \: .
\end{align}

\begin{lemma}[FTRL regret equality]
\label{lemma:ftrl-equality}
Let $\bell_1, \bell_2, \dots, \bell_T \in \R^d$ and $\bx_t \in \argmin_{\bx \in
\R^d} \ H_t(\bx)$ where $H_t(\bx)$ is defined in
\eqref{equation:ftrl-objective}. Then, for any $\bu \in \R^d$,
\begin{equation}
\label{equation:ftrl-equality}
\sum_{t=1}^T \ip{\bell_t}{\bx_t - \bu}
= \phi_{T+1}(\bu) - \min_{\bx \in \R^d} \phi_1(\bx) + H_{T+1}(\bx_{T+1}) - H_{T+1}(\bu) + \sum_{t=1}^T [H_t(\bx_t) - H_{t+1}(\bx_{t+1}) + \ip{\bell_t}{\bx_t}] \: .
\end{equation}
\end{lemma}

\begin{lemma}[FTRL for stochastic optimization]
\label{lemma:regret-bound}
Let $F:\R^d \to \R$ be any differentiable function. Suppose that for all $t \ge
1$, the gradient $\bg_t$ satisfies \eqref{equation:stochastic-gradient} and
\eqref{equation:gradient-bound} and the learning rate $\eta_t$ is a non-negative
$\mathcal{F}_t$-measurable random variable. Assume $\phi_1, \phi_2, \dots$ are convex
differentiable and for all $t \ge 1$ satisfy
\begin{equation}
\label{equation:key-inequality-assumption}
H_t(\bx_t) - H_{t+1}(\bx_{t+1}) + \ip{\bell_t}{\bx_t} \le 0 \: .
\end{equation}
Then, Algorithm~\ref{algorithm:ftrl-rescaled-gradients} satisfies for all $T \ge 0$ and all $\bu \in \R^d$,
\begin{equation}
\label{equation:regret-bound-1}
\Exp[B_{\phi_{T+1}}(\bu, \bx_t)] + \sum_{t=1}^T \Exp\left[ \eta_t \ip{\grad F(\bx_t)}{\bx_t - \bu} \right] \le \Exp\left[ \phi_{T+1}(\bu) \right] - \min_{\bx \in \R^d} \phi_1(\bx) \: .
\end{equation}
In particular,
\begin{equation}
\label{equation:regret-bound-2}
\sum_{t=1}^T \Exp\left[ \eta_t \ip{\grad F(\bx_t)}{\bx_t - \bu} \right] \le \Exp\left[ \phi_{T+1}(\bu) \right] - \min_{\bx \in \R^d} \phi_1(\bx) \: .
\end{equation}
Additionally, if there exists $\bx^* \in \R^d$ such that $\ip{\grad F(\bx)}{\bx
- \bx^*} \ge 0$ for all $\bx \in \R^d$ then
\begin{equation}
\label{equation:regret-bound-3}
\Exp[B_{\phi_{T+1}}(\bx^*, \bx_t)] \le \Exp\left[ \phi_{T+1}(\bx^*) \right] - \min_{\bx \in \R^d} \phi_1(\bx) \: .
\end{equation}
\end{lemma}
\begin{proof}
Lemma~\ref{lemma:ftrl-equality}~and the assumption
\eqref{equation:key-inequality-assumption} imply that
\begin{equation}
\label{lemma:regret-bound-proof-1}
\sum_{t=1}^T \ip{\bell_t}{\bx_t - \bu} \le \phi_{T+1}(\bu) - \min_{\bx \in \R^d} \phi_1(\bx) + H_{T+1}(\bx_{T+1}) - H_{T+1}(\bu) \: .
\end{equation}
Since $\bx_{t+1}$ is a minimizer of $H_{T+1}$, $B_{H_{T+1}}(\bu, \bx_t) =
H_{T+1}(\bu) - H_{T+1}(\bx_{T+1})$. Furthermore, since  $H_{T+1}$ and
$\phi_{T+1}$ differ by a linear function, $B_{H_{T+1}}(\bu, \bx_t) =
B_{\phi_{T+1}}(\bu, \bx_{T+1})$. Therefore, \eqref{lemma:regret-bound-proof-1}
is equivalent to
$$
B_{\phi_{T+1}}(\bu, \bx_{T+1}) + \sum_{t=1}^T \ip{\bell_t}{\bx_t - \bu} \le \phi_{T+1}(\bu) - \min_{\bx \in \R^d} \phi_1(\bx) \: .
$$
Substituting $\bell_t=\eta_t \bg_t$ and taking expectation of both sides yields
$$
\Exp[B_{\phi_{T+1}}(\bu, \bx_{T+1})] + \sum_{t=1}^T \Exp\left[ \eta_t \ip{\bg_t}{\bx_t - \bu} \right] \le \Exp\left[ \phi_{T+1}(\bu) \right] - \min_{\bx \in \R^d} \phi_1(\bx) \: .
$$
We compute $\Exp\left[ \eta_t \ip{\bg_t}{\bx_t - \bu} \right]$ as
\begin{align*}
\Exp\left[ \eta_t \ip{\bg_t}{\bx_t - \bu} \right]
= \Exp\left[ \Exp[\eta_t  \ip{\bg_t}{\bx_t - \bu} ~\middle|~ \mathcal{F}_t ] \right]
= \Exp\left[ \eta_t \ip{\Exp[ \bg_t ~|~ \mathcal{F}_t ]}{\bx_t - \bu} \right]
= \Exp[\eta_t \ip{\grad F(\bx_t)}{\bx_t - \bu}]
\end{align*}
and inequality \eqref{equation:regret-bound-1} follows. Inequality
\eqref{equation:regret-bound-2} follows from \eqref{equation:regret-bound-1} and
the fact that $B_{\phi_{T+1}}(\cdot, \cdot)$ is non-negative. Inequality
\eqref{equation:regret-bound-3} follows from \eqref{equation:regret-bound-1} and
the inequality $\ip{\grad F(\bx_t)}{\bx_t - \bx^*} \ge 0$ which holds by
assumption.
\end{proof}

The assumption \eqref{equation:key-inequality-assumption} might seem strange at
first. Generally speaking, the analysis of FTRL with an arbitrary sequence of
regularizers boils down to proving an upper bound on $H_t(\bx_t) -
H_{t+1}(\bx_{t+1}) + \ip{\bell_t}{\bx_t}$. By adding a suitable constant to the
regularizer $\phi_t$, one can ensure that the upper bound is zero.

The surprising fact is that the regularizer $\phi_t$ we use to construct
Algorithm~\ref{algorithm:ftrl-rescaled-gradients-exponetial-potential} satisfies
\eqref{equation:key-inequality-assumption} under the assumption $\norm{\bell_t}
\le 1$ (see Lemma~\ref{lemma:key-inequality} in
Appendix~\ref{section:regularizer-appendix}) and $\min_{\bx \in \R^d}
\phi_1(\bx) = \phi_1(\boldsymbol{0}) =-1$. These kind of regularizers and this
proof technique were introduced by \citet{Orabona13}.

The inequality \eqref{equation:regret-bound-3} is essential for the proof of the
asymptotic convergence for variationally coherent functions. The idea of using
this Bregman divergence to guarantee convergence for non-strongly convex
functions was pioneered by \citet{DekelGS2010}.

\section{Linearithmic Regularizer}
\label{section:regularizer}

In order to define the sequence of (non-strongly convex) linearithmic regularizers, we define the functions
$\psi^*$ and $\psi$. The function $\psi^*:\R \times (0, \infty) \times
[0,\infty) \to \R$ is defined by
\begin{equation}
\label{equation:psi-star-definition}
\psi^*(\theta, S, Q) = \exp \left( \max_{\beta \in \left[-\frac{1}{2},\frac{1}{2}\right]} \theta \beta - \beta^2 S^2 - Q \right) \: .
\end{equation}
The function $\psi:\R \times (0, \infty) \times [0,\infty) \to \R$ is defined
as the Fenchel conjugate of $\psi^*$ with respect to the first argument,
\begin{equation}
\label{equation:psi-definition}
\psi(x, S, Q) = \sup_{\theta \in \R} \theta x - \psi^*(\theta, S, Q) \: .
\end{equation}
The maximum over $\beta$ in \eqref{equation:psi-star-definition} can be removed
and replaced with an explicit formula
\begin{equation}
\label{equation:explicit-formula-for-psi-star}
\psi^*(\theta, S, Q) =
\begin{cases}
\exp\left( \dfrac{\theta^2}{4S^2} - Q \right) & \text{if $\abs{\theta} \le S^2$,} \\[0.5cm]
\exp\left( \dfrac{\abs{\theta}}{2} - \frac{1}{4} S^2 - Q \right) &  \text{if $\abs{\theta} > S^2$.}
\end{cases}
\end{equation}
Lemma~\ref{lemma:psi-properties} in Appendix~\ref{section:regularizer-appendix}
lists many properties of $\psi^*$ and $\psi$, including an explicit
formula for $\psi$ of the order of $O(|x| S (\ln(|x|+1)+Q+S))$. For now, it suffices to say
that both $\theta \mapsto \psi^*(\theta,S,Q)$ and $x \mapsto \psi(x,S,Q)$ are
even, strictly convex, continuously differentiable, and increasing on $[0, +\infty)$. Furthermore, $\theta \mapsto
\psi^*(\theta,S,Q)$ and $x \mapsto \psi(x,S,Q)$ are Fenchel conjugates of each
other. Their partial derivatives $\theta \mapsto \frac{\partial
\psi^*(\theta,S,Q)}{\partial \theta}$ and $x \mapsto \frac{\partial
\psi(x,S,Q)}{\partial x}$ are continuous bijections from $\R$ to $\R$ that are
inverses of each other.

\paragraph{Definition of $\phi_t$ and its properties}
The regularizer $\phi_t:\R^d \to \R$ is defined in terms $\psi$ as
\begin{align}
\phi_t(\bx) & = \psi(\norm{\bx}, S_{t-1}, Q_{t-1})  && \text{for $t=1,2,\dots$} \label{equation:regularizer} \: ,
\end{align}
where $S_t$ and $Q_t$ are defined in Algorithm~\ref{algorithm:ftrl-rescaled-gradients-exponetial-potential}.
We also define $\phi_t^*:\R^d \to \R$,
\begin{align}
\phi_t^*(\btheta) & = \psi^*(\norm{\btheta}, S_{t-1}, Q_{t-1}) && \text{for $t=1,2,\dots$} \label{equation:potential} \: .
\end{align}
Lemma~\ref{lemma:fenchel-conjugate-of-function-of-norm} in
Appendix~\ref{section:regularizer-appendix} implies that $\phi_t^*$ and $\phi_t$
are Fenchel conjugates of each other. Using the properties of $\psi$ and
$\psi^*$, it is easy to verify that both $\phi^*_t$ and $\phi^*_t$ are strictly
convex and continuously differentiable. The gradient maps $\grad \phi_t:\R^d \to
\R^d$ and $\grad \phi_t^*:\R^d \to \R^d$ are continuous bijections and inverses
of each other.

The sequences $\{S_t\}_{t=1}^\infty$, $\{Q_t\}_{t=1}^\infty$ and
$\{\phi_t(\bx)\}_{t=1}^\infty$ are non-decreasing. Under assumption
\eqref{equation:learning-rates-assumption-1} the sequences are bounded and have
finite limits $S_\infty$, $Q_\infty$ and $\phi_\infty(\bx) = \psi(\norm{\bx},
S_\infty, Q_\infty)$ and these limits are bounded random variables; see
Lemma~\ref{lemma:limits} in Appendix~\ref{section:regularizer-appendix}.

\paragraph{Explicit formulas}
We derive Algorithm~\ref{algorithm:ftrl-rescaled-gradients-exponetial-potential}
as a special case of Algorithm~\ref{algorithm:ftrl-rescaled-gradients} with
sequence of regularizers defined in \eqref{equation:regularizer}. According to the
the definitions of the algorithms, $\bx_t = \argmin_{\bx \in \R^d} \phi_t(\bx) - \ip{\btheta_{t-1}}{\bx}$.
Since $\bx_t$ is a minimizer of $\phi_t(\bx) - \ip{\btheta_{t-1}}{\bx}$, it
satisfies the first order stationarity condition $\grad \phi_t(\bx_t) =
\btheta_{t-1}$. Since $\grad \phi^*_t$ and $\grad \phi_t$ are inverses of each
other, $\bx_t = \grad \phi^*_t(\btheta_{t-1})$.
Formulas \eqref{equation:explicit-formula-for-psi-star} and
\eqref{equation:potential} give an explicit formula
\begin{equation}
\label{equation:regularizer-conjugate-1}
\phi^*_t(\btheta) =
\begin{cases}
\exp\left( \dfrac{\norm{\btheta}^2}{4S_{t-1}^2} - Q_{t-1} \right) &   \text{if $\norm{\btheta} \le S_{t-1}^2$,} \\[0.5cm]
\exp\left( \dfrac{\norm{\btheta}}{2} - \frac{1}{4} S_{t-1}^2 - Q_{t-1} \right) &  \text{if $\norm{\btheta} > S_{t-1}^2$,}
\end{cases}
\end{equation}
from which can compute $\bx_t = \grad \phi^*_t(\btheta_{t-1})$ and derive the
formula on Line~3 of
Algorithm~\ref{algorithm:ftrl-rescaled-gradients-exponetial-potential}.

\section{Proofs of the Main Results}
\label{section:proofs}

In this section, we present the proofs of our main results. For a matter of
readability, we only present the main and most interesting steps here, leaving
the proofs of the technical lemmas to the Appendix. As in
Section~\ref{section:ftrl}, for simplicity of notation, we assume
$\bx_0=\boldsymbol{0}\in \R^d$ and obtain the general results with a simple
translation of the coordinate system.

\subsection{Proof of Theorem~\ref{theorem:variationally-coherent-convergence}}
\label{section:proof-variationally-coherent-convergence}

In the proof, we first show that $B_{\phi_t}(\bx^*, \bx_t)$ converges to a
finite limit almost surely. Then, we show that this limit is 0. In turn, this
will prove the convergence of $\bx_t$ to $\bx^*$, even if $\phi_t$ is not
strongly convex. We will need the following two lemmas, the proofs are in
Appendix~\ref{section:variationally-coherent}.
\begin{lemma}[Convergence of Bregman divergences]
\label{lemma:bregman-convergence}
If $F:\R^d \to \R$ is variationally coherent then there exists a random variable
$B_\infty$ such that $\lim_{t \to \infty} \ B_{\phi_t}(\bx^*, \bx_t) = B_{\infty} < \infty$ almost surely.
\end{lemma}
\begin{lemma}[$\norm{\bx_t-\bx^*}^2$ is squeezed]
\label{lemma:iterates-squeezed}
There exists two random variables $C_1$ and $C_2$ such that with probability one,
$0 < C_1 < C_2 < \infty$ and $C_1 \norm{\bx_t - \bx^*}^2 \le B_{\phi_t}(\bx^*, \bx_t) \le C_2 \norm{\bx_t - \bx^*}^2$.
\end{lemma}

\begin{proof}[Proof of Theorem~\ref{theorem:variationally-coherent-convergence}]
Lemma~\ref{lemma:regret-bound} and $\min_{\bx \in \R^d} \phi_1(\bx) =
-1$ imply that for any $T \ge 0$,
\[
\sum_{t=1}^T \Exp\left[ \eta_t \ip{\grad F(\bx_t)}{\bx_t - \bx^*} \right]
\le 1 + \Exp\left[ \phi_{\infty}(\bx^*) \right] \: .
\]
Since $F$ is variationally coherent, $\eta_t \ip{\grad F(\bx_t)}{\bx_t - \bx^*} \ge 0$.
Monotone convergence theorem implies
\[
\Exp\left[ \sum_{t=1}^\infty \eta_t \ip{\grad F(\bx_t)}{\bx_t - \bx^*} \right]
\le 1 + \Exp\left[ \phi_{\infty}(\bx^*) \right]
< \infty \: .
\]
Therefore,
\begin{equation}
\label{equation:regret}
\sum_{t=1}^\infty \eta_t \ip{\grad F(\bx_t)}{\bx_t - \bx^*} < \infty \qquad \text{almost surely.}
\end{equation}
Moreover, Lemma~\ref{lemma:bregman-convergence} implies that
$\lim_{t \to \infty} B_{\phi_t}(\bx^*,\bx_t) = B_{\infty} < \infty$ almost surely.

Now, we claim that $B_\infty = 0$. Clearly, $B_\infty \ge 0$. Suppose by
contradiction that $B_\infty$ is strictly positive. Then, there exists a random
variable $T_0$ such that $T_0 < \infty$ almost surely and $\frac{1}{2} B_\infty
\le B_{\phi_t}(\bx^*,\bx_t) \le 2 B_\infty$ for all $t \ge T_0$. So,
Lemma~\ref{lemma:iterates-squeezed} implies that $\sqrt{\frac{B_\infty}{2C_2}}
\le \norm{\bx_t - \bx^*} \le \sqrt{\frac{2B_\infty}{C_1}}$ for all $t \ge T_0$.
Now, let
\[
\delta = \inf \left\{ \ip{\grad F(\bx)}{\bx - \bx^*} ~:~ \bx \in \R^d,
\sqrt{\frac{B_\infty}{2C_2}} \le \norm{\bx - \bx^*} \le \sqrt{\frac{2B_\infty}{C_1}}  \right\} \: .
\]
Since $F$ is continuously differentiable, the function $\bx \mapsto \ip{\grad
F(\bx)}{\bx - \bx^*}$ is continuous. The infimum is taken over a compact set
$\{\bx \in \R^d ~:~ \sqrt{\frac{B_\infty}{2C_2}} \le \norm{\bx_t - \bx^*} \le
\sqrt{\frac{2B_\infty}{C_1}} \}$. Therefore, the infimum is attained at some
point $\widetilde{\bx}$ in this set. That is, $\delta = \ip{\grad
F(\widetilde{\bx})}{\widetilde{\bx} - \bx^*}$. Since $B_{\infty} > 0$,
$\widetilde{\bx} \neq \bx^*$ and therefore $\delta > 0$. Thus,
\[
\sum_{t=1}^\infty \eta_t \ip{\grad F(\bx_t)}{\bx_t - \bx^*}
\ge \sum_{t=T_0}^\infty \eta_t \ip{\grad F(\bx_t)}{\bx_t - \bx^*}
\ge \delta \sum_{t=T_0}^\infty \eta_t = \infty \qquad \text{almost surely} \: ,
\]
which contradicts \eqref{equation:regret}.
Thus, $\lim_{t \to \infty} B_{\phi_t}(\bx^*,\bx_t) = 0$ almost surely.
Finally, Lemma~\ref{lemma:iterates-squeezed} implies that  $\norm{\bx_t - \bx^*}$
converges to $0$ almost surely as well.
\end{proof}

\subsection{Proof of Theorem~\ref{theorem:convex-average-non-adaptive}}
\label{section:proof-convex-average-non-adaptive}

Given the results in Lemma~\ref{lemma:regret-bound}, the proof of
Theorem~\ref{theorem:convex-average-non-adaptive} follows from standard
arguments from online convex optimization and online-to-batch
conversion~\citep{Cesa-bianchiCG02}. We only need a technical lemma to upper
bound the values of $\phi_T(\bu)$. Its proof is in
Appendix~\ref{section:convex-case}.

\begin{lemma}[Bound on $S_T,Q_T,\phi_T(\bu)$]
\label{lemma:bound-S-Q-psi-non-adaptive}
Let $\alpha > 1/2$. If $\eta_t = \frac{1}{Gt^\alpha}$,
then, for any $T \ge 0$ and any $\bu \in \R^d$, we have
\begin{align*}
S_T \le \sqrt{5  + \frac{1}{2\alpha-1}} \: , \qquad
Q_T \le \ln \left(5  + \frac{1}{2\alpha-1} \right) \: , \\
\phi_T(\bu) \le \sqrt{5 + \frac{1}{2\alpha - 1}} \norm{\bu} \left[ 2\ln(1 + 2 \norm{\bu}) + 9 \sqrt{5 + \frac{1}{2\alpha -1}} \right] \: .
\end{align*}
\end{lemma}

\begin{proof}[Proof of Theorem~\ref{theorem:convex-average-non-adaptive}]
Lemma~\ref{lemma:regret-bound} and $\min_{\bx \in \R^d} \phi_1(\bx) = -1$ imply
that for any $\bu \in \R^d$,
\[
\sum_{t=1}^T \Exp\left[ \eta_t \ip{\grad F(\bx_t)}{\bx_t - \bu} \right] \le 1 + \Exp\left[\phi_{T+1}(\bu) \right] \: .
\]
Since $F$ is convex, $F(\bx_t) - F(\bu) \le \ip{\grad F(\bx_t)}{\bx_t - \bu}$.
Substituting $\bx^*$ for $\bu$, we have
\[
\sum_{t=1}^T  \Exp\left[ \eta_t (F(\bx_t) - F(\bx^*)) \right] \le 1 + \Exp\left[\phi_{T+1}(\bx^*) \right] \: .
\]
Since $\{\eta_t\}_{t=1}^\infty$ is non-negative decreasing and $F(\bx_t) -
F(\bx^*)$ is non-negative,
\[
\eta_T \sum_{t=1}^T  \Exp\left[ F(\bx_t) - F(\bx^*) \right] \le 1 + \Exp\left[\phi_{T+1}(\bx^*) \right] \: .
\]
By Jensen's inequality, $F(\overline{\bx}_T) \le \frac{1}{T} \sum_{t=1}^T
F(\bx_t)$. Thus, substituting for $\eta_T$ and using
Lemma~\ref{lemma:bound-S-Q-psi-non-adaptive} to upper bound $\phi_{T+1}(\bx^*)$,
we get the stated bound.
\end{proof}

\subsection{Proof of Theorem~\ref{theorem:convex-last-iterate}}
\label{section:proof-convex-last-iterate}

Here, we prove the convergence of the last iterate, extending the approach of
\citet{Orabona20} to FTRL with rescaled gradients. We need the following Lemmas
that are proved in Appendix~\ref{section:last-iterate}.
\begin{lemma}{\citep{Orabona20}}
\label{lemma:last-average}
Let $\eta_1, \eta_2, \dots, \eta_T$ be a non-increasing sequence of non-negative
numbers. Let $q_1, q_2, \dots, q_T$ be non-negative. Then
\[
\eta_T q_T \le \frac{1}{T} \sum_{t=1}^T \eta_t q_t + \sum_{k=1}^{T-1} \frac{1}{k(k+1)} \sum_{t=T-k}^T \eta_t (q_t - q_{T-k})~.
\]
\end{lemma}

\begin{lemma}[Difference of regularizers]
\label{lemma:difference-of-regularizers}
Let $A,T$ be integers such that $1 \le A \le T+1$. Then,
\begin{equation}
\label{equation:difference-of-regularizers}
\Exp[\phi_{T+1} (\bx_A) - \phi_A(\bx_A)]
\le K \sum_{t=A}^T t^{-2\alpha},
\end{equation}
where $K=\frac{1}{2} \exp\left(\frac{1}{4} S \right) + 3\norm{\bx^*} + 5(S + 2) \Exp[2 + \phi_\infty(\bx^*)]$.
\end{lemma}

\begin{lemma}[Interesting sum]
\label{lemma:bound-sum-k}
Let $\alpha> \frac{1}{2}$. Then, $\sum_{k=1}^{T-1} \frac{1}{k(k+1)} \sum_{t=T-k}^T t^{-2\alpha} \le \frac{1}{T} + T^{-2\alpha} + \frac{1}{e(2\alpha - 1)T}$.
\end{lemma}

\begin{lemma}[FTRL partial regret bound]
\label{lemma:partial-regret}
Let $\bell_1, \bell_2, \dots, \bell_T \in \R^d$ and $\bx_t \in \argmin_{\bx \in
\R^d} \ H_t(\bx)$ where $H_t(\bx)$ is defined in
\eqref{equation:ftrl-objective}. Assume that for all $t \ge 1$,
$H_t(\bx_t) - H_{t+1}(\bx_{t+1}) + \ip{\bell_t}{\bx_t} \le 0$.
Then, for any $A \le T$, we have
\[
\sum_{t=A}^T \ip{\bell_t}{\bx_t - \bx_A}
\le \phi_{T+1} (\bx_A) - \phi_A(\bx_A) \: .
\]
\end{lemma}

\begin{proof}[Proof of Theorem~\ref{theorem:convex-last-iterate}]
Starting from Lemma~\ref{lemma:partial-regret}, we substitute $\bell_t = \eta_t
\bg_t$, take expectation of both sides, and use that $\Exp\left[ \eta_t
\ip{\bg_t}{\bx_t - \bx_A} \right] = \Exp\left[ \eta_t \ip{\grad F(\bx_t)}{\bx_t
- \bx_A} \right]$ by assumption \eqref{equation:stochastic-gradient}. We get
\[
\sum_{t=A}^T \Exp\left[ \eta_t \ip{\grad F(\bx_t)}{\bx_t - \bx_A} \right] \le \Exp\left[\phi_{T+1} (\bx_A) - \phi_A(\bx_A) \right] \: .
\]
From Lemma~\ref{lemma:difference-of-regularizers} and convexity of $F$, we obtain
\begin{equation}
\label{equation:tail-regret}
\sum_{t=A}^T \eta_t \Exp\left[  F(\bx_t)  - F(\bx_A) \right]
\le \sum_{t=A}^T \Exp\left[ \eta_t \ip{\grad F(\bx_t)}{\bx_t - \bx_A} \right]
\le \sum_{t=A}^T \Exp\left[\phi_{T+1} (\bx_A) - \phi_A(\bx_A) \right]
\le K \sum_{t=A}^T t^{-2\alpha} \: .
\end{equation}
We now apply Lemma~\ref{lemma:last-average} with $q_t=\Exp[F(\bx_t)] - F(\bx^*)$ and get
\begin{align*}
\eta_T \Exp[F(\bx_T) - F(\bx^*)]
& \le \frac{1}{T} \sum_{t=1}^T \eta_t \Exp[F(\bx_t) - F(\bx^*)] + \sum_{k=1}^{T-1} \frac{1}{k(k+1)} \sum_{t=T-k}^T \eta_t \Exp[F(\bx_t) - F(\bx_{T-k})] \\
& \le \frac{1}{T} \sum_{t=1}^T \eta_t \Exp[F(\bx_t) - F(\bx^*)] + K \sum_{k=1}^{T-1} \frac{1}{k(k+1)} \sum_{t=T-k}^T t^{-2\alpha}
\end{align*}
where in second step we used \eqref{equation:tail-regret} with $A=T-k$. We upper bound the first sum
using convexity of $F$ and Lemma~\ref{lemma:regret-bound} as follows
\begin{align*}
\sum_{t=1}^T \eta_t \Exp[F(\bx_t) - F(\bx^*)]
\le \sum_{t=1}^T \eta_t \Exp[\ip{\grad F(\bx_t)}{\bx_t - \bx^*}]
\le 1 + \Exp[\phi_{T+1}(\bx^*)] \: .
\end{align*}
Finally, using Lemma~\ref{lemma:bound-sum-k}, we have
\begin{equation*}
\eta_T \Exp[F(\bx_T) - F(\bx^*)]
\le \frac{1 + \Exp[\phi_{T+1}(\bx^*)]}{T} + K \left( \frac{1}{T} + T^{-2\alpha} + \frac{1}{e(2\alpha - 1)T} \right) \: .
\end{equation*}
We multiply both sides by $1/\eta_T = G T^{\alpha}$ and get
and use $1/T^{1-\alpha} \ge 1/T^{\alpha}$ and $\phi_{T+1}(\bx^*) \le \phi_{\infty}(\bx^*)$, to get
\[
\Exp[F(\bx_T) - F(\bx^*)] \le G \frac{1 + \Exp[\phi_\infty(\bx^*)]}{T^{1-\alpha}} + G K \left( \frac{2}{T^{1-\alpha}} + \frac{1}{e(2\alpha - 1)T^{\alpha-1}} \right) \: .
\]
Substituting the definition of $K$, using Lemma~\ref{lemma:bound-S-Q-psi-non-adaptive} to upper bound $\phi_{\infty}(\bx^*)$ and over-approximating, we obtain the stated result.
% \begin{align*}
% & \Exp[F(\bx_T)] - F(\bx^*) \\
% & \le \frac{G}{T^{1-\alpha}} \left( 1 + \Exp[\phi_\infty(\bx^*)] +  \left( \frac{1}{2} \exp\left(\frac{1}{4} S \right) + 3\norm{\bx^*} + 5(S + 2) \Exp[2 + \phi_\infty(\bx^*)] \right) \left( 2 + \frac{1}{e(2\alpha - 1)} \right) \right) \\
% & \le \frac{G}{T^{1-\alpha}} \left( \frac{1}{2} \exp\left(\frac{1}{4} S \right) + 3\norm{\bx^*} + 6(S + 2) \Exp[2 + \phi_\infty(\bx^*)] \right) \left( 2 + \frac{1}{e(2\alpha - 1)} \right) \\
% & \le \frac{G}{T^{1-\alpha}} \left( \exp(S) + 3\norm{\bx^*} + 6(S + 2) \Exp[2 + \phi_\infty(\bx^*)] \right) \left( 2 + \frac{1}{e(2\alpha - 1)} \right) \: .
% \end{align*}
% The proof is finished by using Lemma~\ref{lemma:bound-S-Q-psi-non-adaptive} to upper bound $\phi_{\infty}(\bx^*)$.
\end{proof}

\section{Discussions on Limitations and Future Work}
\label{section:conclusions}

We have presented the first algorithm that simultaneously achieve the best known
convergence rate on convex function with bounded stochastic gradients and also
guarantees asymptotic convergence with probability one on variationally coherent
functions. In the following, we want to discuss some limitations and possible
future directions.

\paragraph{Alternative assumptions} \cite{Bottou98} uses a slightly different
set of conditions in the definition of variationally coherent functions. He
assumes that for every $\epsilon > 0$ there exists $\delta > 0$ such that
$\ip{\grad F(\bx)}{\bx - \bx^*} > \delta$ whenever $\norm{\bx - \bx^*} >
\epsilon$ and drops the condition of continuous differentiability. His
assumption is incomparable with ours. Nevertheless, our
Theorem~\ref{theorem:variationally-coherent-convergence} would still hold true
as is, with a minor modification of its proof. The advantage of Bottou's
condition is that Theorem~\ref{theorem:variationally-coherent-convergence}
generalizes to (infinite-dimensional) Hilbert spaces, while our argument is
based on the compactness of balls in $\R^d$.

\paragraph{Additional adaptivity}
It is very natural to ask if further adaptivity is possible. For example, one
could think to use data-dependent learning rates that depends on the sum of the
squared norms of the previous gradients. Indeed, we have an additional result in
Appendix~\ref{section:adaptive-learning-rate} that shows that the function value
evaluate on the average iterate would convergence at a rate of $O(1/T)$ if the
gradients are deterministic and it would match the convergence of
Theorem~\ref{theorem:convex-average-non-adaptive} in the stochastic case.
Moreover, the convergence result on variationally coherent functions would still
hold! However, we were unable to prove the convergence of the last iterate for
these learning rates and we leave it as a future direction of work.

\paragraph{Further applications of FTRL with rescaled gradients}
We firmly believe that FTRL with Rescaled Gradients might have many more
applications that the one presented here. The common knowledge in OCO and
optimization literature is that the degree of freedom of choosing the learning
rates in OMD corresponds to the degree of freedom to choosing time-varying
regularizers in FTRL. However, we have shown here that sometimes \emph{both}
degrees of freedom are necessary. Another example is the general form of the
recently proposed dual-stabilized OMD~\citep{FangHPF20}, that with Legendre
regularizer can be verified being an instantiation of FTRL with rescaled
gradients \emph{and} time-varying regularizers~\citep[Proposition
H.5,][]{FangHPF20}.

\paragraph{The need for bounded stochastic gradients}
Parameter-free algorithms have a fundamental limitation in the fact that the
(stochastic) gradient must be bounded and the bound must be known to the
algorithm, due to the lower bound in \citet{CutkoskyB17}. In the deterministic
case, it is enough to use normalized gradients to avoid the knowledge of the
bound on the gradients, as explained in \citet[Section 3.2.3]{Nesterov04}.
Another approach that would work also in the stochastic setting has been
proposed by \citet{Cutkosky19c}, that showed that it is possible to avoid the
knowledge of the maximum gradient norm, paying an additional
$O(\frac{\norm{\bx^*}^3}{T})$ term in the convergence guarantee. Yet, we do not
know how to extend parameter-free algorithm to non-Lipschitz function, for
example, to smooth functions. Note that in our theorems we proved that $\bx_t$
is bounded, that would imply a bounded gradient even with smooth functions. Yet, it
is unclear how to modify to the current proof to argue that $\bx_t$ are bounded
even in the smooth case. On the other hand, it is important to remember that
assuming smoothness is not a weaker assumption than bounded gradients.

\paragraph{Optimality of the results}
As explained in Section~\ref{section:related-work}, it is unclear if these
results are optimal even in the stochastic convex case. We would need a lower
bound for stochastic convex optimization with bounded gradients for unbounded
domains, that is currently missing. Indeed, all the lower bounds we know assume
a bounded domain. We conjecture that a similar lower bound to the one
\citet{StreeterM12} could be proven for stochastic convex optimization.

\bibliography{biblio}
\bibliographystyle{plainnat_nourl}

\appendix

\pagebreak
\section{Proofs and Lemmas for FTRL with Rescaled Gradients}
\label{section:ftrl-proofs}

\begin{proof}[Proof of Lemma~\ref{lemma:ftrl-equality}]
We cancel $\sum_{t=1}^T \ip{\bell_t}{\bx_t}$ on both sides of \eqref{equation:ftrl-equality}.
We get an equivalent equation
$$
- \sum_{t=1}^T \ip{\bell_t}{\bu}
= \phi_{T+1}(\bu) - \min_{\bx \in \R^d} \phi_1(\bx) + H_{T+1}(\bx_{T+1}) - H_{T+1}(\bu) + \sum_{t=1}^T [H_t(\bx_t) - H_{t+1}(\bx_{t+1})] \: .
$$
The sum $\sum_{t=1}^T [H_t(\bx_t) - H_{t+1}(\bx_{t+1})]$ is a telescopic sum equal to $H_1(\bx_1) - H_{T+1}(\bx_{T+1})$. Therefore, \eqref{equation:ftrl-equality} is equivalent to
$$
- \sum_{t=1}^T \ip{\bell_t}{\bu}
= \phi_{T+1}(\bu) - \min_{\bx \in \R^d} \phi_1(\bx) + H_{T+1}(\bx_{T+1}) - H_{T+1}(\bu) + H_1(\bx_1) - H_{T+1}(\bx_{T+1}) \: .
$$
We cancel common terms and use the fact $\min_{\bx \in \R^d} \phi_1(\bx) = H_1(\bx_1)$ and we get
$$
- \sum_{t=1}^T \ip{\bell_t}{\bu}
= \phi_{T+1}(\bu) - H_{T+1}(\bu) \: ,
$$
which holds true by definition of $H_{T+1}(\bu)$.
\end{proof}

\pagebreak
\section{Properties of the Regularizer}
\label{section:regularizer-appendix}

\begin{lemma}[Fenchel conjugate of a function of $\norm{\cdot}$]
\label{lemma:fenchel-conjugate-of-function-of-norm}
Let $f:\R \to \R \cup \{-\infty, +\infty\}$ be even and let $f^*$ be its Fenchel conjugate. Let $g:\R^d
\to \R \cup \{-\infty, +\infty\}$ be defined as $g(\bx) = f(\norm{\bx})$. The Fenchel conjugate of $g$
satisfies $g^*(\btheta) = f^*(\norm{\btheta})$ for every $\btheta \in \R^d$.
\end{lemma}
\begin{proof}
For any $\btheta \in \R^d$,
\begin{align*}
g^*(\btheta)
& = \sup_{\bx \in \R^d} \ip{\btheta}{\bx} - g(\bx)
= \sup_{\bx \in \R^d} \ip{\btheta}{\bx} - f(\norm{\bx})
= \sup_{\rho \in [0,\infty)} \sup_{\substack{\bz \in \R^d \\ \norm{\bz} = 1}} \ip{\btheta}{\rho \bz} - f(\norm{\rho \bz}) \\
& = \sup_{\rho \in [0, \infty)} \sup_{\substack{\bz \in \R^d \\ \norm{\bz} = 1}} \rho \ip{\btheta}{\bz} - f(\rho)
= \sup_{\rho \in [0, \infty)} \rho \norm{\btheta} - f(\rho)
= \sup_{\rho \in \R} \abs{\rho} \norm{\btheta} - f(\abs{\rho})
= \sup_{\rho \in \R} \rho \norm{\btheta} - f(\abs{\rho}) \\
& = \sup_{\rho \in \R} \rho \norm{\btheta} - f(\rho)
%& \text{(since $f$ is even)} \\
= f^*(\norm{\btheta}),
\end{align*}
where in the second to last equality we used the fact that $f$ is even.
\end{proof}

\begin{lemma}[Properties of Lambert $W$ function]
\label{lemma:lambert-function-properties}
Let $W:[0,\infty) \to [0,\infty)$ be the inverse of the function $f:[0,\infty) \to [0, \infty)$ for $f(x) = x e^x$.
Then, $W$ is a continuous increasing bijection and satisfies
\begin{align*}
W(0) & = 0 \: , &
W(x e^x) & = x \: , &
W(x) e^{W(x)} & = x \: .
\end{align*}
Furthermore, for any $x \in [0,\infty)$,
\begin{equation}
\label{equation:lambert-function-bounds}
\frac{1}{2} \ln(1 + x) \le W(x) \le \ln(1 + x) \: .
\end{equation}
\end{lemma}
\begin{proof}
The function $f:[0,\infty) \to [0,\infty)$, $f(x) = x e^x$, is an increasing
continuous bijection. Thus its inverse is an increasing continuous bijection.
The properties $W(x e^x) = x$ and $W(x) e^{W(x)} = x$ follow from the definition
of inverse. The property $W(0) = 0$ is a special case of $W(x e^x) = 0$ for $x=0$.

Both inequalities in \eqref{equation:lambert-function-bounds} holds for $x = 0$.
It thus suffices to prove them for $x > 0$. The second inequality in
\eqref{equation:lambert-function-bounds} is a special case of a more general
bound proved by \citet{HoorfarH08},
$$
W(x) \le \ln \left( \frac{x+y}{1 + \log y} \right) \qquad \text{for all $x > - \frac{1}{e}$ and all $y > -\frac{1}{e}$} \: ,
$$
for $y=1$.

To prove the first inequality in \eqref{equation:lambert-function-bounds} we
start from $W(x) e^{W(x)} = x$. We take logarithm of both sides and we get $W(x)
= \ln(x / W(x))$. Using the second inequality in
\eqref{equation:lambert-function-bounds}, we have
\[
W(x)
= \ln \left( \frac{x}{W(x)} \right)
\ge \ln \left( \frac{x}{\ln (1+x)} \right) \: .
\]
It remains to prove that for $x > 0$
\[
\ln \left( \frac{x}{\ln (1+x)} \right) \ge \frac{1}{2} \ln(1 + x) \: ,
\]
which is equivalent to
\[
x \ge \ln(1+x) \sqrt{1 + x} \: .
\]
The last inequality holds for $x=0$ with
equality. We take derivatives of both sides. It remains to prove
\[
1 \ge \frac{1 + \frac{1}{2}\ln(1+x)}{\sqrt{1+x}} \qquad \text{for $x \ge 0$.}
\]
The last inequality is equivalent to
\[
\sqrt{1+x} \ge 1 + \ln(\sqrt{1+x}) \qquad \text{for $x \ge 0$.}
\]
Substituting $\sqrt{1+x} = 1 + z$, we need to prove
\[
\ln(1+z) \le z  \qquad \text{for $z \ge 0$.}
\]
The last inequality is holds for $z=0$. We take derivative of both sides. It remains to prove
\[
\frac{1}{1+z} \le 1 \: ,
\]
which clearly holds for all $z \ge 0$.
\end{proof}

\begin{lemma}[Properties of $\psi$ and $\psi^*$]
\label{lemma:psi-properties}
The functions $\psi^*:\R \times (0, \infty) \times [0,\infty) \to \R$ and
$\psi:\R \times (0, \infty) \times [0,\infty) \to \R$ have the following
properties.
\begin{enumerate}
\item \label{lemma:psi-property-1} For any $S > 0$ and any $Q \ge 0$, the function $\theta \mapsto \psi^*(\theta, S, Q)$ is positive, even, continuously differentiable, strictly convex, decreasing on $(-\infty,0]$ and increasing on $[0,+\infty)$.
\item \label{lemma:psi-property-2} For any $S > 0$ and any $Q \ge 0$, the function $x \mapsto \psi(x, S, Q)$ is even, continuously differentiable, strictly convex, decreasing on $(-\infty,0]$ and increasing on $[0,+\infty)$.
\item \label{lemma:psi-property-3} $\psi(x, S, Q)$ is non-decreasing in its second and third argument.
\item \label{lemma:psi-property-4} If $x \ge 0$, $\frac{\partial \psi(x,S,Q)}{\partial x}$ is non-decreasing in $S$ and $Q$.
\item \label{lemma:psi-property-5} For any $S' \ge S > 0$, $Q' \ge Q \ge 0$, $x' \ge x \ge 0$,
$$
\psi(x,S',Q') - \psi(x,S,Q) \le \psi(x',S',Q') - \psi(x',S,Q) \: .
$$

\item \label{lemma:psi-property-6} For any $S \ge 1$, $Q \ge 0$ and any $x > 0$,
\begin{equation}
\label{equation:first-derivate-bounds}
\sqrt{\ln(1+2x^2)}
\le \frac{\partial \psi(x,S,Q)}{\partial x}
\le x \max\left\{2 S^2 \exp(Q), \ 4+\frac{S^2 + 4 Q-4}{\exp\left(\frac{1}{4} S^2 - Q\right)} \right\}
\end{equation}
\item \label{lemma:psi-property-7} For any $S \ge 1$, $Q \ge 0$ and any $x \in \R \setminus \{0, \pm \frac{1}{2}\exp(\frac{1}{4}S^2 - Q)\}$,
\begin{equation}
\label{equation:second-derivate-lower-bound}
\frac{\min\{2, \sqrt{\ln(1 + 2x^2)}\}}{\abs{x} (\frac{1}{2}S + 1)}
\le \frac{\partial^2 \psi(x,S,Q)}{\partial x^2}
\le \max\left\{ 2S^2 \exp(Q), \ 4\exp\left(Q-\frac{1}{4} S^2 \right)\right\} \: .
\end{equation}
\item For any $S > 0$, $Q \ge 0$, any $x \in \R$,
\begin{equation}
\label{equation:psi-explicit-formula}
\psi(x,S,Q) =
\begin{cases}
- \exp(-Q) & \text{if $x=0$} \: , \\[0.5cm]
S \abs{x} \sqrt{2} \dfrac{W(2 \exp(2Q) S^2 x^2) - 1}{\sqrt{W(2 \exp(2Q) S^2 x^2)}} & \text{if $0 < |x| \le \dfrac{1}{2} \exp\left(\frac{1}{4} S^2 - Q\right)$} \: , \\[0.5cm]
2 \abs{x} \ln (2\abs{x}) + \abs{x} (\frac{1}{2} S^2 + Q - 2) & \text{if $\abs{x} > \dfrac{1}{2} \exp\left(\frac{1}{4} S^2 - Q\right)$} \: ,
\end{cases}
\end{equation}
where $W:[0,\infty) \to [0,\infty)$ is the Lambert W-function, i.e., $W$ is the inverse function of $x \mapsto x e^x$.
\item For any $S \ge 1$, $Q \ge 0$, any $x \in \R$,
$$
\psi(x,S,Q) <  S \abs{x} \left[ 2\ln(1 + 2 x^2) + 3Q + 3S \right] \: .
$$
\end{enumerate}
\end{lemma}
\begin{proof}
$ $ % dummy line

\begin{enumerate}[wide,nosep]
\item Equation \eqref{equation:explicit-formula-for-psi-star} implies that
for any $S > 0$, $Q \ge 0$, the function $\theta \mapsto \psi^*(\theta,S,Q)$ is positive and even.
Its derivative is
$$
\frac{\partial \psi^*(\theta, S, Q)}{\partial \theta}
=
\begin{cases}
\frac{\theta}{2 S^2}\exp\left(\dfrac{\theta^2}{4 S^2}-Q\right) & \text{if $\abs{\theta} \le S^2$,} \\[0.5cm]
\frac{1}{2} \sign(\theta) \exp\left(\dfrac{|\theta|}{2}-\frac{1}{4} S^2 - Q\right) & \text{if $\abs{\theta} > S^2$.}
\end{cases}
$$
The function $\theta \mapsto \frac{\partial \psi^*(\theta, S, Q)}{\partial
\theta}$ is a continuous odd increasing bijection from $\R$ to $\R$ that is
negative on $(-\infty, 0)$ and positive on $(0,+\infty)$. Therefore, $\theta
\mapsto \psi^*(\theta, S, Q)$ is strictly convex, decreasing on $(-\infty,0]$
and increasing on $[0,+\infty)$.

\item Since $\theta \mapsto \psi^*(\theta,S,Q)$ is even,
$$
\psi(x, S, Q)
= \sup_{\theta \in \R} \theta x - \psi^*(\theta, S, Q)
= \sup_{\theta \in \R} \theta x - \psi^*(\abs{\theta}, S, Q)
= \sup_{\theta \in \R} \abs{\theta x} - \psi^*(\abs{\theta}, S, Q) \: ,
$$
and thus, the function $x \mapsto \psi(x, S, Q)$ is even as well.

Since $\psi$ and $\psi^*$ are Fenchel conjugates, the functions $x \mapsto \frac{\partial \psi(x,
S, Q)}{\partial x}$ and $\theta \mapsto \frac{\partial \psi^*(\theta, S,
Q)}{\partial \theta}$ are functional inverses of one another. We can express
$\frac{\partial \psi(x, S, Q)}{\partial x}$ as
\begin{equation}
\label{equation:psi-derivative}
\frac{\partial \psi(x, S, Q)}{\partial x}
=\begin{cases}
\sign(x) S \sqrt{2 W(2 \exp(2Q) S^2 x^2)} & \text{if $\abs{x} \le \dfrac{1}{2} \exp\left(\frac{1}{4} S^2 - Q\right)$,} \\[0.5cm]
\sign(x) \left(2 \ln (2\abs{x}) + \frac{1}{2} S^2 + 2 Q\right) & \text{if $\abs{x} > \dfrac{1}{2} \exp\left(\frac{1}{4} S^2 - Q\right)$,}
\end{cases}
\end{equation}
where $W:[0,\infty) \to [0,\infty)$ is the Lambert function that is the inverse
of the function $x \mapsto x e^x$. Using the properties of the Lambert function
(Lemma~\ref{lemma:lambert-function-properties}), it easy to verify that the
function $x \mapsto \frac{\partial \psi(x, S, Q)}{\partial x}$ is a continuous
odd increasing bijection from $\R$ to $\R$ that is negative on $(-\infty, 0)$
and positive on $(0,+\infty)$. Therefore $x \mapsto \psi(x, S, Q)$ is strictly
convex, decreasing on $(-\infty,0]$ and increasing on $[0, +\infty)$.

\item Equation \eqref{equation:explicit-formula-for-psi-star} implies that
$\psi^*(\theta,S,Q)$ is non-increasing both as a function of $S$ and as a
function of $Q$. From the definition \eqref{equation:psi-definition}, we see
that $\psi(x,S,Q)$ is non-decreasing both as a function of $S$ and as a function
of $Q$.

\item The equation \eqref{equation:psi-derivative} implies that for $x \ge 0$,
the function $S \mapsto \frac{\partial \psi(x, S, Q)}{\partial x}$ is
non-decreasing on the interval $(0,\infty)$. Likewise, for $x \ge 0$, the
function $Q \mapsto \frac{\partial \psi(x, S, Q)}{\partial x}$ is non-decreasing
on the interval $[0,\infty)$.

\item Using the previous property,
\begin{align*}
\psi(x,S',Q') - \psi(x,S,Q)
& = \psi(0,S',Q') - \psi(0,S,Q) + \int_0^x \frac{\partial \psi(y,S',Q')}{\partial y} - \frac{\partial \psi(y,S,Q)}{\partial y} dy \\
& \le \psi(0,S',Q') - \psi(0,S,Q) + \int_0^{x'} \frac{\partial \psi(y,S',Q')}{\partial y} - \frac{\partial \psi(y,S,Q)}{\partial y} dy \\
& = \psi(x',S',Q') - \psi(x',S,Q) \: .
\end{align*}

\item First, we prove the lower bound. If $0 < x \le \frac{1}{2} \exp\left(\frac{1}{4} S^2 - Q\right)$ then
\begin{align*}
\frac{\partial \psi(x, S, Q)}{\partial x}
& = \sqrt{2S^2 W(2 \exp(2Q) S^2 x^2)} \\
& \ge \sqrt{2 W(2 x^2)} & \text{(since $S \ge 1$ and $Q \ge 0$)} \\
& \ge \sqrt{\ln (1 + 2x^2)} & \text{(by Lemma~\ref{lemma:lambert-function-properties})} \: .
\end{align*}
If $x > \frac{1}{2} \exp\left(\frac{1}{4} S^2 - Q\right)$ then $\frac{\partial \psi(x, S, Q)}{\partial x} \ge 1$ and
therefore
\begin{align*}
\frac{\partial \psi(x, S, Q)}{\partial x}
& \ge \sqrt{\frac{\partial \psi(x, S, Q)}{\partial x}} \\
& = \sqrt{2 \ln (2 x) + \frac{1}{2} S^2 + 2 Q} \\
& \ge \sqrt{2 \ln (2x) + \frac{1}{2}} & \text{(since $S \ge 1$ and $Q \ge 0$)} \\
& = \sqrt{\ln(4\sqrt{e}x^2)} \\
& \ge \sqrt{\ln (6x^2)} \\
& \ge \sqrt{\ln (1 + 2x^2)} & \text{(since $x \ge 1/2$)} \: .
\end{align*}

For the upper bound, we study the function $\frac{1}{x}\frac{\partial \psi(x, S, Q)}{\partial x}$. For $0\le x\le \frac{1}{2} \exp\left(\frac{1}{4} S^2 - Q\right)$, we have that
\[
\frac{1}{x}\frac{\partial \psi(x, S, Q)}{\partial x}
= \frac{\sqrt{2S^2 W(2 \exp(2Q) S^2 x^2)}}{x}
= \sqrt{2S^2 \cdot 2 \exp(2Q) S^2}\frac{\sqrt{W(y)}}{\sqrt{y}}
\le 2 S^2 \exp(Q)
\]
where $y=2 \exp(2Q) S^2 x^2 \ge 0$ and since $\frac{\sqrt{W(y)}}{\sqrt{y}}\le 1$ for all $y\ge0$.
In the same way, for $x\ge \frac{1}{2} \exp\left(\frac{1}{4} S^2 - Q\right)$
\[
\frac{1}{x}\frac{\partial \psi(x, S, Q)}{\partial x}
= \frac{2 \ln (2x) + \frac{1}{2} S^2 + 2 Q}{x}
\le \frac{4x - 2 + \frac{1}{2} S^2 + 2 Q}{x}
\le 4+\frac{S^2 + 4 Q-4}{\exp\left(\frac{1}{4} S^2 - Q\right)} \: .
\]
The final upper bound is the sum of the upper bounds.

\item Taking derivative of \eqref{equation:psi-derivative}, we obtain
the second partial derivative,
\[
\frac{\partial^2 \psi(x, S, Q)}{\partial x^2}
=\begin{cases}
\dfrac{\sqrt{2S^2 W(2 \exp(2Q) S^2 x^2)}}{\abs{x} (W(2 \exp(2Q) S^2 x^2)+1)} & \text{if $0 < \abs{x} \le \dfrac{1}{2} \exp\left(\frac{1}{4} S^2 - Q\right)$,} \\[0.5cm]
\dfrac{2}{\abs{x}} & \text{if $\abs{x} > \dfrac{1}{2} \exp\left(\frac{1}{4} S^2 - Q\right)$.}
\end{cases}
\]
The lower bound holds, if $\abs{x} > \frac{1}{2} \exp\left(\frac{1}{4} S^2 - Q\right)$.
If $0 < \abs{x} \le \frac{1}{2} \exp\left(\frac{1}{4} S^2 - Q\right)$, we have
\begin{align*}
\frac{\partial^2 \psi(x, S, Q)}{\partial x^2}
& = \frac{\sqrt{2S^2 W(2 \exp(2Q) S^2 x^2)}}{\abs{x} (W(2 \exp(2Q) S^2 x^2)+1)} \\
& \ge \frac{\sqrt{2S^2 W(2 \exp(2Q) S^2 x^2)}}{\abs{x} (W(2 \exp(2Q) S^2 \frac{1}{4} \exp(\frac{1}{2} S^2 - 2 Q))+1)} & \text{(since $W(\cdot)$ is increasing)} \\
& = \frac{\sqrt{2S^2 W(2 \exp(2Q) S^2 x^2)}}{\abs{x} (W(\frac{1}{2}S^2 \exp(\frac{1}{2} S^2))+1)} \\
& = \frac{\sqrt{2S^2 W(2 \exp(2Q) S^2 x^2)}}{\abs{x} (\frac{1}{2}S^2 + 1)}  \\
& \ge \frac{\sqrt{S^2 \ln(1 + 2 \exp(2Q) S^2 x^2)}}{\abs{x} (\frac{1}{2}S^2 + 1)} & \text{(by Lemma~\ref{lemma:lambert-function-properties})} \\
& \ge \frac{S \sqrt{\ln(1 + 2x^2)}}{\abs{x} (\frac{1}{2}S^2 + 1)} & \text{(since $S \ge 1$ and $Q \ge 0$)} \\
& \ge \frac{\sqrt{\ln(1 + 2x^2)}}{\abs{x} (\frac{1}{2}S + 1)} & \text{(since $S \ge 1$)} \: .
\end{align*}

For the upper bound, if $0 < x \le \frac{1}{2} \exp\left(\frac{1}{4} S^2 - Q\right)$, we have
\begin{align*}
\frac{\partial^2 \psi(x, S, Q)}{\partial x^2}
& = \frac{\sqrt{2S^2 W(2 \exp(2Q) S^2 x^2)}}{\abs{x} (W(2 \exp(2Q) S^2 x^2)+1)}
= \sqrt{2S^2 \cdot 2 \exp(2 Q) S^2} \frac{\sqrt{W(y)}}{\sqrt{y} (W(y)+1)} \\
& = 2S \exp(2 Q) \frac{\sqrt{W(y)}}{\sqrt{y} (W(y)+1)} \: ,
\end{align*}
where $y = 2 \exp(2 Q) S^2 x^2$. It is possible to verify that
$\frac{\sqrt{W(y)}}{\sqrt{y} (W(y)+1)}\le 1$ for $y\ge 0$. Hence, the first
expression of the max follows. The second expression is immediate when $x \ge
\frac{1}{2} \exp\left(\frac{1}{4} S^2 - Q\right)$.

\item
We compute $\psi(0,S,Q)$ as
\[
\psi(0,S,Q)
= \sup_{\theta \in \R} - \psi^*(\theta,S,Q)
= - \inf_{\theta \in \R} \psi^*(\theta,S,Q)
= -\psi^*(0,S,Q)
= - \exp(-Q) \: .
\]
If $0 < \abs{x} \le \frac{1}{2} \exp\left(\frac{1}{4} S^2 - Q\right)$, from
\eqref{equation:psi-derivative} and the fact $x \mapsto \psi(x, S, Q)$ is even
and continuous, we have that
\begin{align*}
\psi(x,S,Q)
& = \psi(0,S,Q) + \int_0^{\abs{x}} \frac{\partial \psi(y, S, Q)}{\partial y} dy
= \psi(0,S,Q) + \int_0^{\abs{x}} \sqrt{2S^2 W(2 \exp(2Q) S^2 y^2)} dy \\
& = \psi(0,S,Q) + S \abs{x}\sqrt{2} \frac{W(2 \exp(2Q) S^2 x^2)-1}{\sqrt{W(2 \exp(2Q) S^2 x^2)}} + \frac{\sqrt{2 S^2}}{\sqrt{2 \exp(2Q) S^2}} \\
& = S \abs{x}\sqrt{2} \frac{W(2 \exp(2Q) S^2 x^2)-1}{\sqrt{W(2 \exp(2Q) S^2 x^2)}} \: .
\end{align*}

If $\abs{x} > \frac{1}{2} \exp\left(\frac{1}{4} S^2 - Q\right)$, let
$x_0=\frac{1}{2} \exp\left(\frac{1}{4} S^2 - Q\right)$. From
\eqref{equation:psi-derivative} and the fact $x \mapsto \psi(x, S, Q)$ is even
and continuous, we have that
\begin{align*}
\psi(x,S,Q)
& = \psi(x_0,S,Q) + \int_{x_0}^{\abs{x}} \frac{\partial \psi(y, S, Q)}{\partial y} dy
= \psi(x_0,S,Q) + \int_{x_0}^{\abs{x}} 2 \ln (2y) + \frac{1}{2} S^2 + 2 Q dy \\
& = \psi(x_0,S,Q) + 2|x| \ln (2\abs{x}) - 2 x_0 \ln(2x_0) + \left(\frac{1}{2} S^2 + 2 Q - 2 \right) (\abs{x} - x_0) \\
& = \psi(x_0,S,Q) + 2|x| \ln (2\abs{x}) - 2 x_0 \left( \frac{1}{4}S^2 - Q \right) + \left(\frac{1}{2} S^2 + 2 Q - 2 \right) (\abs{x} - x_0) \: .
\end{align*}
We express $\psi(x_0,S,Q)$ as
\begin{align*}
\psi(x_0,S,Q)
& = S x_0 \sqrt{2} \frac{W(2 \exp(2Q) S^2 x_0^2)-1}{\sqrt{W(2 \exp(2Q) S^2 x_0^2)}}
= S x_0 \sqrt{2} \frac{W(2 \exp(2Q) S^2 \frac{1}{4} \exp(\frac{1}{2} S^2 - 2Q)) - 1}{\sqrt{W(2 \exp(2Q) S^2 \frac{1}{4} \exp(\frac{1}{2} S^2 - 2Q))}} \\
& = S x_0 \sqrt{2} \frac{W(\frac{1}{2} S^2 \exp(\frac{1}{2} S^2)) - 1}{\sqrt{W(\frac{1}{2} S^2 \exp(\frac{1}{2} S^2))}}
= S x_0 \sqrt{2} \frac{\frac{1}{2} S^2 - 1}{\sqrt{\frac{1}{2} S^2}}
= 2 x_0 \left( \frac{1}{2} S^2 - 1 \right) \: .
\end{align*}
Hence, if $\abs{x} > \frac{1}{2} \exp\left(\frac{1}{4} S^2 - Q\right)$, we have
\begin{align*}
\psi(x,S,Q)
& = 2 x_0 \left( \frac{1}{2} S^2 - 1 \right) + 2|x| \ln (2\abs{x}) - 2 x_0 \left( \frac{1}{4}S^2 - Q \right) + \left(\frac{1}{2} S^2 + 2 Q - 2 \right) (\abs{x} - x_0) \\
& = 2|x| \ln (2\abs{x}) + \abs{x} \left(\frac{1}{2} S^2 + 2 Q - 2 \right) \: .
\end{align*}

\item If $\abs{x} \le \frac{1}{2} \exp(\frac{1}{4}S^2 - Q)$, then
\begin{align*}
\psi(x,S,Q)
& = S \abs{x} \sqrt{2} \frac{W(2 \exp(2Q) S^2 x^2) - 1}{\sqrt{W(2 \exp(2Q) S^2 x^2)}} \\
& \le S \abs{x} \sqrt{2} W(2 \exp(2Q) S^2 x^2) & \text{(since $\tfrac{W-1}{\sqrt{W}} \le W$ for $W \ge 0$)} \\
& \le S \abs{x} \sqrt{2} \ln(1 + 2 \exp(2Q) S^2 x^2) & \text{(Lemma~\ref{lemma:lambert-function-properties})} \\
& \le S \abs{x} \sqrt{2} \ln(\exp(2Q) S^2 + 2 \exp(2Q) S^2 x^2) & \text{(since $S \ge 1$ and $Q \ge 0$)} \\
& \le S \abs{x} \sqrt{2} \left[\ln(1 + 2 x^2) + \ln(\exp(2Q)S^2) \right] \\
& = S \abs{x} \sqrt{2} \left[ \ln(1 + 2 x^2) + 2Q + 2 \ln S \right] \\
& \le S \abs{x} \sqrt{2} \left[ \ln(1 + 2 x^2) + 2Q + 2S \right] \\
& < S \abs{x} \left[ 2\ln(1 + 2 x^2) + 3Q + 3S \right] \: .
\end{align*}
If $\abs{x} > \frac{1}{2} \exp(\frac{1}{4}S^2 - Q)$, then
\begin{align*}
\psi(x,S,Q)
& = 2 \abs{x} \ln (2\abs{x}) + \abs{x} \left( \frac{1}{2} S^2 + Q - 2 \right)
\le 2 S \abs{x} \ln (2\abs{x}) + \abs{x} \left( \frac{1}{2} S^2 + Q - 2 \right) \\
& \le 2 S \abs{x} \ln (1 + 2x^2) + \abs{x} \left( \frac{1}{2} S^2 + Q - 2 \right)
\le 2 S \abs{x} \ln (1 + 2x^2) + \abs{x} \left( \frac{1}{2} S^2 + SQ \right) \\
& = S \abs{x} \left[ 2 \ln (1 + 2x^2) + \frac{1}{2} S + Q \right]
< S \abs{x} \left[ 2 \ln (1 + 2x^2) + 3S + 3Q \right] \: .
\end{align*}
\end{enumerate}
\end{proof}

\begin{lemma}[Useful inequality]
\label{lemma:useful-inequality}
Let $a_0 \in \R$ and let $a_1, a_2, \dots, a_T \in [0,\infty)$.
Let $f:[a_0, \sum_{t=0}^T a_t] \to \R$ be a non-increasing function. Then,
\[
\sum_{t=1}^T a_t f\left(a_0 + \sum_{i=1}^{t} a_i\right) \le \int_{a_0}^{\sum_{t=0}^T a_t} f(x) dx \: .
\]
\end{lemma}
\begin{proof}
Denote by $s_t=\sum_{i=0}^{t} a_i$ for $t=0,1,2,\dots,T$ and note that $s_0 \le s_1 \le \dots \le s_T$. Then, for any $t=1,2,\dots,T$,
\[
a_t f\left(a_0+ \sum_{i=1}^{t} a_i\right) = a_t f(s_t) = \int_{s_{t-1}}^{s_t} f(s_t) d x \le \int_{s_{t-1}}^{s_t} f(x) dx \: .
\]
Summing over $t=1,2,\dots, T$, we have the stated bound.
\end{proof}

\begin{lemma}[Bound on $Q_T$]
\label{lemma:bound-Q}
For any $T \ge 0$, $Q_T \le 2 \ln S_T$.
\end{lemma}
\begin{proof}
We use Lemma~\ref{lemma:useful-inequality} with $f(x)=1/x$ and $a_0=4$ and $a_t=\norm{\bell_t}^2$. We have
\begin{align*}
Q_T
= \sum_{t=1}^T a_t f\left(a_0 + \sum_{i=1}^{t} a_i\right)
\le \int_{a_0}^{\sum_{t=0}^T a_t} f(x) dx
= \int_4^{S_T^2} f(x) dx
\le \int_1^{S_T^2} f(x) dx
= 2 \ln S_T \: .
\end{align*}
\end{proof}

\begin{lemma}[Limits $S_\infty$, $Q_\infty$, $\psi_\infty$]
\label{lemma:limits}
For any $\bx \in \R^d$, the sequences $\{S_t\}_{t=0}^\infty$,
$\{Q_t\}_{t=0}^\infty$, $\{\phi_t(\bx)\}_{t=1}^\infty$ are non-decreasing.
Furthermore, the assumption \eqref{equation:learning-rates-assumption-1} implies
that the sequences have finite limits
\begin{align*}
S_\infty & = \lim_{t \to \infty} S_t && \text{almost surely} \: , \\
Q_\infty & = \lim_{t \to \infty} Q_t && \text{almost surely} \: , \\
\phi_{\infty}(\bx) & = \lim_{t \to \infty} \phi_t(\bx) = \psi \left(\norm{\bx}, S_{\infty}, Q_{\infty} \right) && \text{almost surely}
\end{align*}
and $S_\infty, Q_\infty, \phi_{\infty}(\bx)$ are bounded random variables.
\end{lemma}
\begin{proof}
According to the definition of Algorithm~\ref{algorithm:ftrl-rescaled-gradients-exponetial-potential},
\begin{align}
S_t & = \sqrt{4 + \sum_{i=1}^t \norm{\bell_i}^2} = \sqrt{4 + \sum_{i=1}^t \eta_i^2 \norm{\bg_i}^2} \: , \\
Q_t & = \sum_{i=1}^t \frac{\norm{\bell_i}^2}{S_i^2} = \sum_{i=1}^t \frac{\eta_i^2 \norm{\bg_i}^2}{S_i^2} \: .
\end{align}
Clearly, the sequences $\{S_t\}_{t=0}^\infty$, $\{Q_t\}_{t=0}^\infty$ are
non-decreasing and satisfy $S_t \ge 2$ and $Q_t \ge 0$. Assumption
\eqref{equation:learning-rates-assumption-1} implies that $S_t < \sqrt{4 +
\gamma}$. Therefore, the limit $S_\infty$ exists, is finite and $2 \le S_\infty
\le \sqrt{4 + \gamma}$. Thus, the random variable $S_\infty$ is bounded. By
Lemma~\ref{lemma:bound-Q}, $Q_t \le 2 \ln S_t < \ln (4 + \gamma)$. Therefore,
the limit $Q_\infty$ exists, finite, and $0 \le Q_\infty \le \ln (4 + \gamma)$.
Thus, the random variable $Q_\infty$ is bounded.

By Lemma~\ref{lemma:psi-properties}, $\psi(\norm{\bx}, S_{t-1}, Q_{t-1}) \le
\psi(\norm{\bx}, S_t, Q_t)$. In other words, $\{\phi_t(\bx)\}_{t=1}^\infty$ is a
non-decreasing sequence. Lemma~\ref{lemma:psi-properties} also implies that
$\psi(x, S_\infty, Q_\infty)$ is continuous as function on $\R \times (0,
\infty) \times [0, \infty)$. Therefore,
$$
\lim_{t \to \infty} \phi_t(\bx)
= \lim_{t \to \infty} \psi(\norm{\bx}, S_{t-1}, Q_{t-1})
= \psi \left(\norm{\bx}, \lim_{t \to \infty} S_{t-1}, \lim_{t \to \infty} Q_{t-1} \right)
= \psi(\norm{\bx}, S_\infty, Q_\infty) \: .
$$
Lemma~\ref{lemma:psi-properties} also implies that
$$
\psi(\norm{\bx}, 2, 0)
\le \psi(\norm{\bx}, S_\infty, Q_\infty)
\le \psi \left(\norm{\bx}, \sqrt{4 + \gamma}, \ln (4 + \gamma) \right) \: .
$$
Therefore, the random variable $\phi_\infty(\bx) = \psi(\norm{\bx}, S_\infty,
Q_\infty)$ is bounded.
\end{proof}

For simplicity, we prove the next lemma using first principles, but it is also possible to observe that $\beta_t$ itself is the output of a certain FTRL algorithm over a constrained set with strongly convex losses.
\begin{lemma}[Key inequality]
\label{lemma:key-inequality}
Let $H_t$ be defined by \eqref{equation:ftrl-objective} and $\psi_t$ be defined
by \eqref{equation:regularizer}. If $\norm{\bell_t} \le 1$ then
\begin{equation}
\label{equation:key-inequality}
H_t(\bx_t) - H_{t+1}(\bx_{t+1}) + \ip{\bell_t}{\bx_t} \le 0 \: .
\end{equation}
\end{lemma}
\begin{proof}
Since $\phi_t^*$ is the Fenchel conjugate of $\phi_t$, $H_t(\bx_t) = -
\phi^*_t(\btheta_{t-1})$, $\phi^*_{t+1}(\btheta_t) = - \phi^*_t(\btheta_{t-1})$
and $\bx_t = \grad \phi^*_t(\btheta_{t-1})$. Therefore, the left-hand side of
\eqref{equation:key-inequality} equals to
$$
\phi^*_{t+1}(\btheta_t) - \phi^*_t(\btheta_{t-1}) + \ip{\bell_t}{\grad \phi^*_t(\btheta_{t-1})} \: .
$$
Since
\begin{align*}
\phi^*_t(\btheta_{t-1}) & =
\begin{cases}
\exp\left( \dfrac{\norm{\btheta}^2}{4S_{t-1}^2} - Q_{t-1} \right) &   \text{if $\norm{\btheta} \le S_{t-1}^2$,} \\[0.5cm]
\exp\left( \dfrac{\norm{\btheta}}{2} - \frac{1}{4} S_{t-1}^2 - Q_{t-1} \right) &  \text{if $\norm{\btheta} > S_{t-1}^2$,}
\end{cases}
\: ,
\end{align*}
the gradient $\grad \phi^*_t(\btheta_{t-1})$ can be expressed as
$$
\grad \phi^*_t(\btheta_{t-1}) = \bbeta_t \phi^*_t(\btheta_{t-1})
$$
where
\begin{equation}
\label{equation:beta}
\bbeta_t
=
\begin{cases}
\dfrac{\btheta_{t-1}}{2 S_{t-1}^2} & \text{if $\norm{\btheta_{t-1}} \le S_{t-1}^2$,} \\[0.5cm]
\dfrac{\btheta_{t-1}}{2 \norm{\btheta_{t-1}}} & \text{if $\norm{\btheta_{t-1}} > S_{t-1}^2$.}
\end{cases}
\end{equation}
Therefore, the left-hand side of \eqref{equation:key-inequality} equals to
$$
\phi^*_{t+1}(\btheta_t) - \phi^*_t(\btheta_{t-1}) (1 - \ip{\bell_t}{\bbeta_t}) \: .
$$
Since $\norm{\bbeta_t} \le 1/2$ and $\norm{\bell_t} \le 1$ and therefore
$\ip{\bbeta_t}{\bell_t} \in [-\frac{1}{2}, \frac{1}{2}]$. Since $1-x \ge
\exp(-x-x^2)$ for any $x \in [-\frac{1}{2}, \frac{1}{2}]$ and $\phi_t^*(\cdot)$
is non-negative, the last expression can be upper bounded as
\begin{align*}
\phi^*_{t+1}(\btheta_t) - \phi^*_t(\btheta_{t-1}) (1 - \ip{\bell_t}{\bbeta_t})
& \le \phi^*_{t+1}(\btheta_t) - \phi^*_t (\btheta_{t-1}) \exp\left(- \ip{\bell_t}{\bbeta_t} - (\ip{\bell_t}{\bbeta_t})^2 \right) \\
& \le \phi^*_{t+1}(\btheta_t) - \phi^*_t (\btheta_{t-1}) \exp\left(- \ip{\bell_t}{\bbeta_t} - \norm{\bell_t}^2 \norm{\bbeta_t}^2 \right) \: .
\end{align*}
It remains to show that
$$
\phi^*_{t+1}(\btheta_t) - \phi^*_t (\btheta_{t-1}) \exp\left(- \ip{\bell_t}{\bbeta_t} - \norm{\bell_t}^2 \norm{\bbeta_t}^2 \right) \le 0 \: ,
$$
which is equivalent to
$$
\phi^*_{t+1}(\btheta_t) \le \phi^*_t(\btheta_{t-1}) \exp\left(- \ip{\bell_t}{\bbeta_t} - \norm{\bbeta_t}^2 \norm{\bell_t}^2 \right)  \: .
$$
We express $\phi^*_{t+1}(\btheta_t)$ and $\phi^*_t(\btheta_{t-1})$ using an
explicit formula
$$
\phi^*_t(\btheta_{t-1}) = \exp\left( \ip{\btheta_{t-1}}{\bbeta_t} - \norm{\bbeta_t}^2 S_{t-1}^2 - Q_{t-1} \right) \: ,
$$
where $\bbeta_t$ is defined by \eqref{equation:beta}. We take logarithm of both
sides and get an equivalent inequality
$$
\ip{\btheta_t}{\bbeta_{t+1}} - \norm{\bbeta_{t+1}}^2 S_t^2 - Q_t \le \ip{\btheta_{t-1}}{\bbeta_t} - \norm{\bbeta_t}^2 S_{t-1}^2 - Q_{t-1} - \ip{\bell_t}{\bbeta_t} -  \norm{\bell_t}^2 \norm{\bbeta_t}^2 \: .
$$
Using the definition $Q_t$ and $Q_{t-1}$, this is equivalent to
$$
\ip{\btheta_t}{\bbeta_{t+1}} - \norm{\bbeta_{t+1}}^2 S_t^2 \le \ip{\btheta_{t-1}}{\bbeta_t} - \norm{\bbeta_t}^2 S_{t-1}^2 + \frac{\norm{\bell_t}^2}{S_t^2} - \ip{\bell_t}{\bbeta_t} -  \norm{\bell_t}^2 \norm{\bbeta_t}^2 \: .
$$
Using the definition of $S_t$ and $S_{t-1}$, this is equivalent to
$$
\ip{\btheta_t}{\bbeta_{t+1}} - \norm{\bbeta_{t+1}}^2 S_t^2 \le \ip{\btheta_{t-1}}{\bbeta_t} - \norm{\bbeta_t}^2 S_t^2 + \frac{\norm{\bell_t}^2}{S_t^2} - \ip{\bell_t}{\bbeta_t} \: .
$$
Using the definition of $\btheta_{t-1}$ and $\btheta_t$, this is equivalent to
\begin{equation}
\label{lemma:key-inequality-proof}
\ip{\btheta_t}{\bbeta_{t+1}} - \norm{\bbeta_{t+1}}^2 S_t^2 \le \ip{\btheta_t}{\bbeta_t} - \norm{\bbeta_t}^2 S_t^2 + \frac{\norm{\bell_t}^2}{S_t^2} \: .
\end{equation}
We prove \eqref{lemma:key-inequality-proof} by considering several cases.

\textbf{Case $\norm{\btheta_{t-1}} > S_{t-1}^2$, $\norm{\btheta_t} > S_t^2$.}
In this case, $\bbeta_t = \frac{\btheta_{t-1}}{2 \norm{\btheta_{t-1}}}$, $\bbeta_{t+1} = \frac{\btheta_t}{2 \norm{\btheta_t}}$
and $\norm{\bbeta_t} = \norm{\bbeta_{t+1}}=1/2$. The inequality \eqref{lemma:key-inequality-proof} becomes
$$
\ip{\btheta_t}{\bbeta_{t+1}} - \frac{1}{4} S_t^2 \le \ip{\btheta_t}{\bbeta_t} - \frac{1}{4} S_t^2 + \frac{\norm{\bell_t}^2}{S_t^2} \: .
$$
Multiplying by $2$, cancelling common terms and rearranging terms, we get an equivalent inequality
$$
\ip{\btheta_t}{\bbeta_{t+1} - \bbeta_t} \le \frac{\norm{\bell_t}^2}{S_t^2} \: .
$$
The last inequality follows since
\begin{align*}
\ip{\btheta_t}{\bbeta_{t+1} - \bbeta_t}
& = \frac{1}{2} \ip{\btheta_t}{\frac{\btheta_t}{\norm{\btheta_t}} - \frac{\btheta_{t-1}}{\norm{\btheta_{t-1}}}} \\
& = \frac{1}{2} \ip{\btheta_{t-1} - \bell_t}{\frac{\btheta_{t-1} - \bell_t}{\norm{\btheta_{t-1} - \bell_t}} - \frac{\btheta_{t-1}}{\norm{\btheta_{t-1}}}} \\
& = \frac{1}{2} \norm{\btheta_{t-1} - \bell_t} - \frac{1}{2} \norm{\btheta_{t-1}} + \frac{1}{2} \ip{\bell_t}{\frac{\btheta_{t-1}}{\norm{\btheta_{t-1}}}}  \\
& = \frac{1}{2} \frac{\ip{\bell_t}{\btheta_{t-1}}}{\norm{\btheta_{t-1}}} + \frac{1}{2} \frac{\norm{\bell_t}^2 - 2 \ip{\bell_t}{\btheta_{t-1}}}{ \norm{\btheta_{t-1}} + \norm{\btheta_{t-1} - \bell_t}} \\
& = \frac{1}{2} \ip{\bell_t}{\btheta_{t-1}} \frac{\norm{\btheta_{t-1} - \bell_t} - \norm{\btheta_t}}{\norm{\btheta_{t-1}}(\norm{\btheta_{t-1}} + \norm{\btheta_{t-1} - \bell_t})} + \frac{1}{2} \frac{\norm{\bell_t}^2}{\norm{\btheta_{t-1}} + \norm{\btheta_{t-1} - \bell_t}} \\
& \le \frac{1}{2} \abs{\ip{\bell_t}{\btheta_{t-1}}} \frac{\abs{\norm{\btheta_{t-1} - \bell_t} - \norm{\btheta_t}}}{\norm{\btheta_{t-1}}(\norm{\btheta_{t-1}} + \norm{\btheta_{t-1} - \bell_t})} + \frac{1}{2} \frac{\norm{\bell_t}^2}{\norm{\btheta_{t-1}} + \norm{\btheta_{t-1} - \bell_t}} \\
& \le \frac{1}{2} \norm{\bell_t} \norm{\btheta_{t-1}} \frac{\norm{\bell_t}}{\norm{\btheta_{t-1}}(\norm{\btheta_{t-1}} + \norm{\btheta_{t-1} - \bell_t})} + \frac{1}{2} \frac{\norm{\bell_t}^2}{\norm{\btheta_{t-1}} + \norm{\btheta_{t-1} - \bell_t}} \\
& = \frac{ \norm{\bell_t}^2}{\norm{\btheta_{t-1}} + \norm{\btheta_{t-1} - \bell_t}} \\
& \le \frac{\norm{\bell_t}^2}{S_{t-1}^2 + S_t^2} \\
& \le \frac{\norm{\bell_t}^2}{S_t^2} \: ,
\end{align*}
where we substituted for $\bbeta_t$ and $\bbeta_{t+1}$, used that, by
definition, $\btheta_t = \btheta_{t-1} - \bell_t$, made some algebraic
manipulation, used triangle inequality in the form $\abs{\norm{\ba} -
\norm{\bb}} \le \norm{\ba - \bb}$, Cauchy-Schwarz inequality and the assumptions
$\norm{\btheta_t} > S_t^2$, $\norm{\btheta_{t-1}} > S_{t-1}^2$ and, finally,
used that $S_{t-1}^2 > 0$.

\textbf{Case $\norm{\btheta_{t-1}} \le S_{t-1}^2$.} In this case, we have $\bbeta_t=\frac{\btheta_{t-1}}{2S_{t-1}^2}$.
Let
$$
\bbeta'_{t+1}= \frac{\btheta_t}{2 S_t^2} = \argmax_{\bbeta \in \R^d} \ip{\btheta_t}{\bbeta} - \norm{\bbeta}^2 S_t^2 \: .
$$
The left-hand side of \eqref{lemma:key-inequality-proof} is upper bounded as
$$
\ip{\btheta_t}{\bbeta_{t+1}} - \norm{\bbeta_{t+1}}^2 S_t^2 \le \ip{\btheta_t}{\bbeta'_{t+1}} - \norm{\bbeta_{t+1}'}^2 S_t^2 \: ,
$$
since $\bbeta_{t+1}$ and $\bbeta'_{t+1}$ are the constrained and unconstrained
maximizers of $\bbeta \mapsto \ip{\btheta_t}{\bbeta} - \norm{\bbeta}^2 S_t^2$
respectively. Thus it suffices to prove
$$
\ip{\btheta_t}{\bbeta'_{t+1}} - \norm{\bbeta'_{t+1}}^2 S_t^2 \le \ip{\btheta_t}{\bbeta_t} - \norm{\bbeta_t}^2 S_t^2 + \frac{\norm{\bell_t}^2}{S_t^2} \: .
$$
Substituting $\btheta_t = 2S_t^2 \bbeta'_{t+1}$, we get
$$
2S_t^2 \norm{\bbeta'_{t+1}}^2 - \norm{\bbeta'_{t+1}}^2 S_t^2 \le 2S_t^2 \ip{\bbeta'_{t+1}}{\bbeta_t} - \norm{\bbeta_t}^2 S_t^2 + \frac{\norm{\bell_t}^2}{S_t^2} \: ,
$$
which is equivalent to
$$
2S_t^2 \norm{\bbeta'_{t+1} - \bbeta_t}^2 \le \frac{\norm{\bell_t}^2}{S_t^2} \: .
$$
The last inequality follows from
\begin{align*}
\norm{\bbeta'_{t+1} - \bbeta_t}^2
& = \norm{\frac{\btheta_{t-1} - \bell_t}{2S_t^2} - \frac{\btheta_{t-1}}{2S_{t-1}^2}}^2 \\
& = \frac{1}{4} \norm{\frac{\btheta_{t-1} - \bell_t}{S_t^2} - \frac{\btheta_{t-1}}{S_{t-1}^2}}^2 \\
& = \frac{1}{4} \norm{\btheta_{t-1} \left(\frac{1}{S_t^2} - \frac{1}{S_{t-1}^2}\right) - \frac{\bell_t}{S_t^2}}^2 \\
& \le \frac{1}{4} \left( \norm{\btheta_{t-1} \left(\frac{1}{S_t^2} - \frac{1}{S_{t-1}^2}\right)} + \frac{\norm{\bell_t}}{S_t^2} \right)^2 \\
& = \frac{1}{4} \left( \norm{\btheta_{t-1}} \left(\frac{1}{S_{t-1}^2} - \frac{1}{S_t^2}\right) + \frac{\norm{\bell_t}}{S_t^2} \right)^2 \\
&\le \frac{1}{2} \norm{\btheta_{t-1}}^2 \left(\frac{1}{S_{t-1}^2} - \frac{1}{S_t^2}\right)^2 + \frac{1}{2} \frac{\norm{\bell_t}^2}{S_t^4} \\
& = \frac{1}{2} \norm{\btheta_{t-1}}^2 \frac{(S_t^2 - S_{t-1}^2)^2}{S_{t-1}^4 S_t^4}  + \frac{1}{2} \frac{\norm{\bell_t}^2}{S_t^4} \\
& = \frac{1}{2} \norm{\btheta_{t-1}}^2 \frac{\norm{\bell_t}^4}{S_{t-1}^4 S_t^4}  + \frac{1}{2} \frac{\norm{\bell_t}^2}{S_t^4} \\
& \le \frac{1}{2} S_{t-1}^4 \frac{\norm{\bell_t}^4}{S_{t-1}^4 S_t^4}  + \frac{1}{2} \frac{\norm{\bell_t}^2}{S_t^4} \\
& = \frac{\norm{\bell_t}^2}{S_t^4} \: ,
\end{align*}
where we substituted for $\bbeta_t$ and $\bbeta'_{t+1}$, used that $\btheta_t =
\btheta_{t-1} - \bell_t$, used triangle inequality, the inequality $0 < S_{t-1}
\le S_t$, the inequality $(a+b)^2 \le 2a^2 + 2b^2$, the fact $S_t = S_{t-1} +
\norm{\bell_t}^2$ and, finally, the inequality $\norm{\btheta_{t-1}} \le S_{t-1}^2$.

\textbf{Case $\norm{\btheta_{t-1}} > S_{t-1}^2$ and $\norm{\btheta_t} \le
S_t^2$.} In this case, we have
$\bbeta_t=\frac{\btheta_{t-1}}{2\norm{\btheta_{t-1}}}$ and $\bbeta_{t+1} =
\frac{\btheta_t}{2S_t^2}$. Since $\btheta_{t-1} = \btheta_t + \bell_t$ , we can
upper bound $\norm{\btheta_{t-1}}$ as
$$
\norm{\btheta_{t-1}} = \norm{\btheta_t + \bell_t} \le \norm{\btheta_t} + \norm{\bell_t} \le  S_t^2 + \norm{\bell_t} \: .
$$
Similarly, since $S_t^2 = S_{t-1}^2 + \norm{\bell_t}^2$ and $\norm{\bell_2} \le 1$, we can lower bound $\norm{\btheta_{t-1}}$ as
$$
\norm{\btheta_{t-1}} > S_{t-1}^2 = S_t^2 - \norm{\bell_t}^2 \ge S_t^2 - \norm{\bell_t} \: .
$$
Therefore,
\begin{equation}
\label{equation:key-inequality-proof-2}
\abs{\norm{\btheta_{t-1}} - S_t^2} \le \norm{\bell_t} \: .
\end{equation}
In order to prove \eqref{lemma:key-inequality-proof}, we substitute $\btheta_t =
2S_t^2 \bbeta_{t+1}$ in it and get
$$
2S_t^2 \norm{\bbeta_{t+1}}^2 - \norm{\bbeta_{t+1}}^2 S_t^2 \le 2S_t^2 \ip{\bbeta_{t+1}}{\bbeta_t} - \norm{\bbeta_t}^2 S_t^2 + \frac{\norm{\bell_t}^2}{S_t^2} \: ,
$$
which is equivalent to
$$
2S_t^2 \norm{\bbeta_{t+1} - \bbeta_t}^2 \le \frac{\norm{\bell_t}^2}{S_t^2} \: .
$$
The last inequality follows from
\begin{align*}
\norm{\bbeta_{t+1} - \bbeta_t}^2
& = \norm{\frac{\btheta_{t-1} - \bell_t}{2S_t^2} - \frac{\btheta_{t-1}}{2\norm{\btheta_{t-1}}}}^2 \\
& = \frac{1}{4} \norm{\frac{\btheta_{t-1} - \bell_t}{S_t^2} - \frac{\btheta_{t-1}}{\norm{\btheta_{t-1}}}}^2 \\
& = \frac{1}{4} \norm{ \btheta_{t-1} \left( \frac{1}{S_t^2} - \frac{1}{\norm{\btheta_{t-1}}} \right) - \frac{\bell_t}{S_t^2}}^2  \\
& \le \frac{1}{4} \left( \norm{ \btheta_{t-1} \left( \frac{1}{S_t^2} - \frac{1}{\norm{\btheta_{t-1}}} \right)} + \frac{\norm{\bell_t}}{S_t^2} \right)^2  \\
& \le \frac{1}{2} \norm{\btheta_{t-1}}^2 \left(\frac{1}{S_t^2} - \frac{1}{\norm{\btheta_{t-1}}}\right)^2 + \frac{1}{2}\frac{\norm{\bell_t}^2}{S_t^4} \\
& = \frac{1}{2} \norm{\btheta_{t-1}}^2 \frac{(\norm{\btheta_{t-1}} - S_t^2)^2}{S_t^4 \norm{\btheta_{t-1}}^2} + \frac{1}{2}\frac{\norm{\bell_t}^2}{S_t^4} \\
& \le \frac{1}{2} \norm{\btheta_{t-1}}^2 \frac{\norm{\bell_t}^2}{S_t^4 \norm{\btheta_{t-1}}^2} + \frac{1}{2}\frac{\norm{\bell_t}^2}{S_t^4} \\
& = \frac{\norm{\bell_t}^2}{S_t^4} \: ,
\end{align*}
where we substituted for $\bbeta_t$ and $\bbeta_{t+1}$, used that $\btheta_t =
\btheta_{t-1} - \bell_t$, used triangle inequality, the inequality $(a+b)^2 \le
2a^2 + 2b^2$, and, finally, the inequality
\eqref{equation:key-inequality-proof-2}.
\end{proof}

\pagebreak
\section{Proofs for Section~\ref{section:proof-variationally-coherent-convergence}}
\label{section:variationally-coherent}

In this section, we prove Lemma~\ref{lemma:iterates-squeezed} and Lemma~\ref{lemma:bregman-convergence}. For its proof, we first need the following technical lemma.

\begin{lemma}[Convergence of non-negative supermartingales]
\label{lemma:supermartingale-convergence}
Let $Y_1, Y_2, \dots$ be a non-negative supermartingale with respect to a
filtration $\mathcal{F}_1, \mathcal{F}_2, \dots$, that is, $Y_t \ge 0$, $\Exp[Y_t] <
\infty$ and $\Exp[Y_{t+1} ~|~ \mathcal{F}_t] \le Y_t$. Then, there exists
a random variable $Y$ such that
$$
\lim_{t \to \infty} Y_t  = Y < \infty \qquad \text{almost surely.}
$$
\end{lemma}
The proof of Lemma~\ref{lemma:supermartingale-convergence} can be found e.g. in
\citet[Theorem~10.8.5, page~373]{Resnick99}.

We can now prove Lemma~\ref{lemma:bregman-convergence}.
\begin{proof}[Proof of Lemma~\ref{lemma:bregman-convergence}]
Lemma~\ref{lemma:key-inequality} implies that
\begin{equation}
\label{equation:bregman-convergence-proof-1}
- H_{t+1}(\bx_{t+1}) \le - H_t(\bx_t) - \ip{\bell_t}{\bx_t} \le 0 \: .
\end{equation}
By definition of $H_t(\cdot)$ and $H_{t+1}(\cdot)$,
\begin{equation}
\label{equation:bregman-convergence-proof-2}
H_{t+1}(\bx^*) - \phi_{t+1}(\bx^*) = H_{t}(\bx^*) - \phi_{t}(\bx^*) + \ip{\bell_t}{\bx^*}  \: .
\end{equation}
We sum \eqref{equation:bregman-convergence-proof-1} and \eqref{equation:bregman-convergence-proof-2} and get
\[
H_{t+1}(\bx^*) - H_{t+1}(\bx_{t+1}) - \phi_{t+1}(\bx^*)
\le H_{t}(\bx^*) - H_t(\bx_t) - \phi_{t}(\bx^*) - \ip{\bell_t}{\bx_t - \bx^*} \: .
\]
Since $\bx_{T+1}$ is the minimizer of $H_{T+1}$, $H_{t+1}(\bx^*) -
H_{t+1}(\bx_{t+1}) = B_{H_{t+1}}(\bx^*, \bx_{t+1}) = B_{\phi_{t+1}}(\bx^*,
\bx_{t+1})$. Similarly, $H_t(\bx^*) - H_t(\bx_t) = B_{\phi_t}(\bx^*, \bx_t)$.
Therefore,
\[
B_{\phi_{t+1}}(\bx^*, \bx_{t+1}) - \phi_{t+1}(\bx^*)
\le B_{\phi_{t}}(\bx^*, \bx_t) - \phi_{t}(\bx^*) - \ip{\bell_t}{\bx_t - \bx^*} \: .
\]
We add $\phi_{\infty}(\bx^*)$ to both sides and we get
\[
B_{\phi_{t+1}}(\bx^*, \bx_{t+1}) + \phi_\infty(\bx^*) - \phi_{t+1}(\bx^*)
\le B_{\phi_t}(\bx^*, \bx_t) + \phi_\infty(\bx^*) - \phi_t(\bx^*) - \ip{\bell_t}{\bx_t - \bx^*} \: .
\]
Let $Y_t = B_{\phi_t}(\bx^*, \bx_t) + \phi_\infty(\bx^*) - \phi_t(\bx^*)$. We have
\[
Y_{t+1} \le Y_t - \ip{\bell_t}{\bx_t - \bx^*} \: .
\]
Equivalently,
\[
Y_{t+1} \le Y_t - \eta_t \ip{\bg_t}{\bx_t - \bx^*} \: .
\]
Taking conditional expectation of both sides,
\[
\Exp[Y_{t+1} ~|~ \mathcal{F}_t] \le Y_t - \eta_t \ip{\grad F(\bx_t)}{\bx_t - \bx^*} \: .
\]
Since $F$ is variationally coherent,
\[
\Exp[Y_{t+1} ~|~ \mathcal{F}_t] \le Y_t \: .
\]
Note that $Y_t$ is non-negative, since $B_{\phi_t}(\cdot, \cdot)$ is
non-negative and $\phi_t(\bx^*) \le \phi_\infty(\bx^*)$. Thus $Y_1, Y_2, \dots$
is a non-negative supermartingale. By
Lemma~\ref{lemma:supermartingale-convergence}, $Y_t$ converges almost surely to
a finite limit. Since $\phi_\infty(\bx^*) - \phi_t(\bx^*) \ge 0$ and
$\{\phi_\infty(\bx^*) - \phi_t(\bx^*)\}_{t=1}^\infty$ is non-increasing,
$\lim_{t \to \infty} (\phi_\infty(\bx^*) - \phi_t(\bx^*))$ exists almost surely.
Therefore, $B_{\phi_t}(\bx^*, \bx_t) = Y_t - \phi_\infty(\bx^*) + \phi_t(\bx^*)$
has a limit almost surely.
\end{proof}

To prove Lemma~\ref{lemma:iterates-squeezed}, we need the following lemmas.
\begin{lemma}[Second derivative bounds]
\label{lemma:second-derivate-bounds}
Let $I\subseteq \R$ be an open interval. Let $f:I \to \R$ be a function such
that $f''(x)$ exists almost everywhere and $f''(x)$ is continuous almost
everywhere. Let $g,h:I \to \R$ be twice continuously differentiable. Suppose,
for almost all $x \in I$,
\begin{equation}
\label{equation:second-derivate-bounds}
g''(x) \le f''(x) \le h''(x) \: .
\end{equation}
Then, for any $u,v \in I$, there exists $z_1, z_2$ between $u$ and $v$ such that
$$
\frac{1}{2} g''(z_1)(u-v)^2 \le f(u) - f(v) - (u-v)f'(v) \le \frac{1}{2} h''(z_2)(u-v)^2 \: .
$$
\end{lemma}
\begin{proof}
The functions $g''(\cdot)$, $h''(\cdot)$ are continuous and therefore bounded on
any closed interval $J \subseteq I$. Similarly, $f''(\cdot)$ is continuous
almost everywhere and by assumption \eqref{equation:second-derivate-bounds}
bounded on any closed interval. Therefore, $g''(\cdot)$, $h''(\cdot)$,
$f''(\cdot)$ are Riemann integrable on any closed interval $J \subseteq I$.
Integrating \eqref{equation:second-derivate-bounds}, we get
$$
\int_v^s g''(x) dx \le \int_v^s f''(x) dx \le \int_v^s h''(x) dx \qquad \text{for any $s \in I$.}
$$
Fundamental theorem of calculus implies that
$$
g'(s) - g'(v) \le f'(s) - f'(v) \le h'(s) - h'(v) \qquad \text{for any $s \in I$.}
$$
Integrating one more time, we get
$$
\int_v^u g'(s) - g'(v) ds \le \int_v^u f'(s) - f'(v) ds \le \int_v^u h'(s) - h'(v) ds
$$
All integrals exists as Riemann integrals, since $f'(\cdot)$, $g'(\cdot)$, $h'(\cdot)$
are necessarily continuous. Fundamental theorem of calculus implies that
$$
g(u) - g(v) - g'(v)(u-v) \le f(u) - f(v) - (u-v)f'(v) \le h(u) - h(v) - h'(v)(u-v) \: .
$$
The lemma follows by applying Taylor's theorem to $g$ and $h$.
\end{proof}

\begin{lemma}[Hessian of radially symmetric functions]
\label{lemma:radially-symmetric-hessian}
Let $f:\R \to \R$ and $g:\R^d \to \R$ be defined as $g(\bx) = f(\norm{\bx})$. If
$f$ is twice differentiable at $\norm{\bx}$ and $\norm{\bx} > 0$ then
\[
\min\left\{f''(\norm{\bx}), \frac{f'(\norm{\bx})}{\norm{\bx}}\right\} I
\preceq \grad^2 g(\bx)
\preceq \max\left\{f''(\norm{\bx}), \frac{f'(\norm{\bx})}{\norm{\bx}}\right\} I \: .
\]
\end{lemma}
\begin{proof}
We need to prove that
\[
\forall \bu \in \R^d \qquad \norm{\bu}^2 \min\left(f''(\norm{\bx}), \frac{f'(\norm{\bx})}{\norm{\bx}}\right) \le \bu^\top \grad^2 g(\bx) \bu \le \norm{\bu}^2 \max\left(f''(\norm{\bx}), \frac{f'(\norm{\bx})}{\norm{\bx}}\right) \: .
\]
The gradient of $g(\bx)$ is $\grad g(\bx) = \frac{\bx}{\norm{\bx}} f'(\norm{\bx})$.
The Hessian of $g(\bx)$ is
\[
\grad^2 g(\bx) = f''(\norm{\bx}) \frac{\bx \bx^\top}{\norm{\bx}^2} + f'(\norm{\bx}) \left(\frac{I}{\norm{\bx}} - \frac{\bx \bx^\top}{\norm{\bx}^3}\right) \: .
\]
In order to upper and lower bound $\bu^\top \grad^2 g(\bx) \bu$,
we decompose the vector $\bu \in \R^d$ as $\bu = \alpha \bx + \bv$ where
$\bv$ is orthogonal to $\bx$. For convenience, let $\beta =
\frac{f''(\norm{\bx})}{\norm{\bx}^2} - \frac{f'(\norm{\bx})}{\norm{\bx}^3}$ and
$\gamma = \frac{f'(\norm{\bx})}{\norm{\bx}}$. We can express $\bu^\top \grad^2
g(\bx) \bu$ as
\begin{align}
\bu^\top \grad^2 g(\bx) \bu
& = \bu^\top \grad^2 (\beta \bx \bx^\top + \gamma I) \bu \notag \\
& = (\alpha \bx + \bv)^\top \grad^2 (\beta \bx \bx^\top + \gamma I) (\alpha \bx + \bv) \notag \\
& = \alpha^2 \beta \norm{\bx}^4 + \alpha^2 \gamma \norm{\bx}^2 + \gamma \norm{\bv}^2 \notag \\
& = \alpha^2 \left(\frac{f''(\norm{\bx})}{\norm{\bx}^2} - \frac{f'(\norm{\bx})}{\norm{\bx}^3}\right) \norm{\bx}^4 + \frac{f'(\norm{\bx})}{\norm{\bx}} (\alpha^2 \norm{\bx}^2 + \norm{\bv}^2) \notag \\
& = \alpha^2 f''(\norm{\bx}) \norm{\bx}^2 + \frac{f'(\norm{\bx})}{\norm{\bx}} \norm{\bv}^2 \label{lemma:radially-symmetric-hessian-proof-1} \: .
\end{align}
Let $A = \min\left(f''(\norm{\bx}), \frac{f'(\norm{\bx})}{\norm{\bx}}\right)$
and $B = \max\left(f''(\norm{\bx}), \frac{f'(\norm{\bx})}{\norm{\bx}}\right)$.
We can upper and lower bound
\[
A(\alpha^2 \norm{\bx}^2 + \norm{\bv}^2) \le \alpha^2 f''(\norm{\bx}) \norm{\bx}^2 + \frac{f'(\norm{\bx})}{\norm{\bx}} \norm{\bv}^2 \le B(\alpha^2 \norm{\bx}^2 + \norm{\bv}^2) \: .
\]
Since $\bx$ and $\bv$ are orthogonal, $\alpha^2 \norm{\bx}^2 + \norm{\bv}^2 = \norm{\bu}^2$. Thus,
\[
A \norm{\bu}^2  \le \bu^\top \grad^2 g(\bx) \bu \le B \norm{\bu}^2 \: .
\]
\end{proof}

\begin{lemma}[Bounds on Bregman divergence]
\label{lemma:bregman-bounds}
For any $t \ge 1$, there exists $\widetilde{x}_t$ between $\norm{\bx^*}$ and $\norm{\bx_t}$ such that
\[
\frac{\norm{\bx^* - \bx_t}^2}{S_{t-1} + 2} \min\left\{ 1, \frac{1}{\widetilde{x}_t} \right\}
\le B_{\phi_t}(\bx^*, \bx_t)
\le 2 \norm{\bx^* - \bx_t}^2 (S_{t-1}^2 + Q_{t-1})\exp(Q_{t-1})  \: .
\]
\end{lemma}
\begin{proof}
Let us define $f:\R \to \R$, $\psi_t:\R \to \R$ and $\bz_t:\R \to \R^d$ as
\begin{align*}
f_t(\alpha) & = \phi_t(\alpha \bx^* + (1-\alpha)\bx_t) \: , \\
\psi_t(x) & = \psi(x, S_{t-1}, Q_{t-1}) \: , \\
\bz_t(\alpha) & = \alpha \bx^* + (1-\alpha) \bx_t \: .
\end{align*}
By definition of Bregman divergence,
\[
B_{\phi_t}(\bx^*, \bx_t) = f_t(1) - f_t(0) - f_t'(0) \: .
\]
We we will use Lemma~\ref{lemma:second-derivate-bounds} to lower and upper bound
the right-hand side. In order to apply the lemma we need upper and lower bounds
on $f_t''(\alpha)$. Since $\phi_t(\bx) = \psi_t(\norm{\bx})$,
\[
f_t''(\alpha) = (\bx^* - \bx_t)^\top \grad^2 \phi_t(\alpha \bx^* + (1-\alpha)\bx_t) (\bx^* - \bx_t) = (\bx^* - \bx_t)^\top \grad^2 \phi_t(\bz_t(\alpha)) (\bx^* - \bx_t)
\]
Therefore, by Lemma~\ref{lemma:radially-symmetric-hessian},
\[
\norm{\bx^* - \bx_t}^2 \min\left\{ \psi_t''(\norm{\bz_t(\alpha)}), \frac{\psi_t'(\norm{\bz_t(\alpha)})}{\norm{\bz_t(\alpha)}}\right\}
\le f(\alpha)
\le \norm{\bx^* - \bx_t}^2 \max\left\{ \psi_t''(\norm{\bz_t(\alpha)}), \frac{\psi_t'(\norm{\bz_t(\alpha)})}{\norm{\bz_t(\alpha)}}\right\} \: .
\]
It remains to lower bound $\min\left\{ \psi_t''(\norm{\bz_t(\alpha)}), \frac{\psi_t'(\norm{\bz_t(\alpha)})}{\norm{\bz_t(\alpha)}}\right\}$
and upper bound $\max\left\{ \psi_t''(\norm{\bz_t(\alpha)}), \frac{\psi_t'(\norm{\bz_t(\alpha)})}{\norm{\bz_t(\alpha)}}\right\}$.
Note that $\psi_t'(x) = \frac{\partial \psi(x,S,Q)}{\partial x}$
and $\psi_t''(x) = \frac{\partial^2 \psi(x,S,Q)}{\partial x^2}$.
We use Lemma~\ref{lemma:psi-properties} to bound these quantities.
A lower bound follows from
\begin{align*}
& \min\left\{ \psi_t''(\norm{\bz_t(\alpha)}), \frac{\psi_t'(\norm{\bz_t(\alpha)})}{\norm{\bz_t(\alpha)}}\right\} \\
& \ge \min\left\{ \frac{2}{\norm{\bz_t(\alpha)}(\frac{1}{2}S_{t-1} + 1)}, \frac{\sqrt{\ln(1 + 2 \norm{\bz_t(\alpha)}^2)}}{\norm{\bz_t(\alpha)}(\frac{1}{2}S_{t-1} + 1)}, \frac{\sqrt{\ln(1 + 2 \norm{\bz_t(\alpha)}^2)}}{\norm{\bz_t(\alpha)}} \right\} \\
& = \frac{\min\left\{2, \sqrt{\ln(1 + 2 \norm{\bz_t(\alpha)}^2)} \right\}}{\norm{\bz_t(\alpha)}(\frac{1}{2}S_{t-1} + 1)} & \text{(since $S_{t-1} \ge 1$)} \\
& \ge \frac{1}{\frac{1}{2}S_{t-1} + 1}  \min\left\{ 1, \frac{1}{\norm{\bz_t(\alpha)}} \right\} & \text{(by case analysis)} \: .
\end{align*}
An upper bound follows from
\begin{align*}
& \max\left\{ \psi_t''(\norm{\bz_t(\alpha)}), \frac{\psi_t'(\norm{\bz_t(\alpha)})}{\norm{\bz_t(\alpha)}}\right\} \\
& \le \max\left\{ 2S_{t-1}^2 \exp(Q_{t-1}),  4\exp\left(Q_{t-1} -\frac{1}{4} S_{t-1}^2 \right), 4 + \frac{S_{t-1}^2 + 4 Q_{t-1} - 4}{\exp\left(\frac{1}{4} S_{t-1}^2 - Q_{t-1}\right)} \right\} \\
& \le \max\left\{ 4S_{t-1}^2 \exp(Q_{t-1}), 4 + \frac{S_{t-1}^2 + 4 Q_{t-1} - 4}{\exp\left(\frac{1}{4} S_{t-1}^2 - Q_{t-1}\right)} \right\} \\
& \le \max\left\{ 4S_{t-1}^2 \exp(Q_{t-1}), 4 + (S_{t-1}^2 + 4 Q_{t-1} - 4)\exp(Q_{t-1}) \right\} \\
& \le \max\left\{ 4S_{t-1}^2 \exp(Q_{t-1}), (S_{t-1}^2 + 4 Q_{t-1})\exp(Q_{t-1}) \right\} \\
& \le 4(S_{t-1}^2 + Q_{t-1})\exp(Q_{t-1}) \: .
\end{align*}

The second derivative $f_t''(\alpha)$ exists almost everywhere and $f''_t(\alpha)$ is continuous almost everywhere.
Lemma~\ref{lemma:second-derivate-bounds} implies that there exists $\alpha^* \in [0,1]$ such that
\begin{align*}
\frac{1}{2} \frac{\norm{\bx^* - \bx_t}^2}{\frac{1}{2}S_{t-1} + 1} \min\left\{ 1, \frac{1}{\norm{\bz_t(\alpha^*)}} \right\} \le f_t(1) - f_t(0) - f_t'(0) \le \frac{1}{2} \norm{\bx^* - \bx_t}^2 4(S_{t-1}^2 + Q_{t-1})\exp(Q_{t-1})  \: .
\end{align*}
Let $\widetilde{x}_t = \norm{\bz_t(\alpha^*)}$. The number $\widetilde{x}_t$ lies
between $\norm{\bx^*}$ and $\norm{\bx_t}$, since $\bz(\alpha^*)$ is a convex
combination of $\alpha \bx^* + (1-\alpha) \bx_t$. Thus,
\begin{align*}
\frac{\norm{\bx^* - \bx_t}^2}{S_{t-1} + 2} \min\left\{ 1, \frac{1}{\widetilde{x}_t} \right\} \le f_t(1) - f_t(0) - f_t'(0) \le 2 \norm{\bx^* - \bx_t}^2 (S_{t-1}^2 + Q_{t-1})\exp(Q_{t-1})  \: .
\end{align*}
\end{proof}

\begin{lemma}[Bounds on iterates]
\label{lemma:iterate-bound}
For any $t \ge 1$,
\begin{equation}
\label{equation:iterate-bound}
\norm{\bx_t} \le \max\left\{\norm{\bx^*} + \sqrt{(S_{t-1} + 2)B_{\phi_t}(\bx^*, \bx_t)}, \ 2 \norm{\bx^*}, \ 4(S_{t-1} + 2) B_{\phi_t}(\bx^*, \bx_t) \right\} \: .
\end{equation}
In particular,
\[
\sup_{t=1,2,\dots} \norm{\bx_t} < \infty \qquad \text{almost surely.}
\]
\end{lemma}
\begin{proof}
Lemma~\ref{lemma:bregman-bounds} implies
\begin{equation}
\label{equation:iterate-bound-proof}
\norm{\bx^* - \bx_t}^2 \le (S_{t-1} + 2) \max\left\{ 1, \widetilde{x}_t \right\} B_{\phi_t}(\bx^*, \bx_t) \: .
\end{equation}
for some $\widetilde{x}_t$ between $\norm{\bx^*}$ and $\norm{\bx_t}$.

In order to prove \eqref{equation:iterate-bound}, we consider three cases.
If $\norm{\bx_t} \le 2 \norm{\bx^*}$ there is nothing to prove. If $\widetilde{x}_t \le 1$ then
$$
\norm{\bx_t} \le \norm{\bx^*} + \norm{\bx^* - \bx_t} \le \norm{\bx^*} + \sqrt{(S_\infty + 2) B_{\phi_t}(\bx^*, \bx_t)},
$$
where in the last step we use \eqref{equation:iterate-bound-proof}.
It remains to verify \eqref{equation:iterate-bound} when
$\norm{\bx_t} > 2 \norm{\bx^*}$ and $\widetilde{x}_t > 1$. We have
\[
\norm{\bx_t}
\le \frac{\norm{\bx_t}^2}{\widetilde{x}_t}
\le \frac{4(\norm{\bx_t} - \norm{\bx^*})^2}{\widetilde{x}_t}
\le \frac{4\norm{\bx_t - \bx^*}^2}{\widetilde{x}_t}
\le 4(S_\infty + 2) B_{\phi_t}(\bx^*, \bx_t),
\]
where in the last step we use \eqref{equation:iterate-bound-proof}.

Lemma~\ref{lemma:bregman-convergence} implies that $\sup_{t=1,2,\dots}
B_{\phi_t}(\bx^*, \bx_t) = B < \infty$. Lemma~\ref{lemma:limits}
implies that $S_{t-1} \le S_\infty < \infty$. Therefore, inequality
\eqref{equation:iterate-bound} implies that
$$
\sup_{t=1,2,\dots} \norm{\bx_t} \le \max\left\{\norm{\bx^*} + \sqrt{(S_\infty + 2)B}, \ 2 \norm{\bx^*}, \ 4(S_\infty + 2) B \right\} \: .
$$
\end{proof}

We are now ready to prove Lemma~\ref{lemma:iterates-squeezed}.
\begin{proof}[Proof of Lemma~\ref{lemma:iterates-squeezed}]
Let $\sup_{t=1,2,\dots} \norm{\bx_t} = X$. Lemma~\ref{lemma:iterate-bound} implies $X < \infty$
almost surely. Lemma~\ref{lemma:bregman-bounds} implies that
$$
\frac{\norm{\bx^* - \bx_t}^2}{S_\infty + 2} \min\left\{ 1, \frac{1}{X} \right\}
\le B_{\phi_t}(\bx^*, \bx_t)
\le 2 \norm{\bx^* - \bx_t}^2 (S_\infty^2 + Q_\infty)\exp(Q_\infty)  \: .
$$
The lemma follows by defining
\begin{align*}
C_1 & = \frac{1}{S_\infty + 2} \min\left\{ 1, \frac{1}{X} \right\} \: , \\
C_2 & = 2(S_\infty^2 + Q_\infty)\exp(Q_\infty) \: .
\end{align*}
\end{proof}

\pagebreak
\section{Proofs for Section~\ref{section:proof-convex-average-non-adaptive}}
\label{section:convex-case}

\begin{proof}[Proof of Lemma~\ref{lemma:bound-S-Q-psi-non-adaptive}]
The bound on $S_T$ follows from
\begin{align*}
S_T^2
& = 4 + \sum_{t=1}^T \norm{\bell_t}^2 
= 4 + \sum_{t=1}^T \eta_t^2 \norm{\bg_t}^2 
= 4 + \sum_{t=1}^T \frac{\norm{\bg_t}^2}{G^2 t^{2\alpha}} 
\le 4 + \sum_{t=1}^T t^{-2\alpha} \\
& = 5 + \sum_{t=2}^T t^{-2\alpha} 
\le 5 + \int_{1}^{T} x^{-2 \alpha} dx 
= 5 + \frac{1 - T^{1-2\alpha}}{2 \alpha-1} 
\le 5 + \frac{1}{2 \alpha-1},
\end{align*}
where in the first inequality we used $\norm{\bg_t} \le G$.
The bound on $Q_T \le 2 \ln S_T$ follows from Lemma~\ref{lemma:bound-Q}.
The bound on $\phi_{T+1}(\bu)$ follows from
\begin{align*}
\phi_T(\bu)
& = \psi(\norm{\bu}, S_{T-1}, Q_{T-1}) \\
& \leq S_{T-1} \norm{\bu} \left[ 2\ln(1 + 2 \norm{\bu}) + 3Q_{T-1} + 3S_{T-1} \right] & \text{(Lemma~\ref{lemma:psi-properties})} \\
& \le S_{T-1} \norm{\bu} \left[ 2\ln(1 + 2 \norm{\bu}) + 6 \ln S_{T-1} + 3S_{T-1} \right] \\
& \le S_{T-1} \norm{\bu} \left[ 2\ln(1 + 2 \norm{\bu}) + 9 S_{T-1} \right] \\
& \le \sqrt{5 + \frac{1}{2\alpha - 1}} \norm{\bu} \left[ 2\ln(1 + 2 \norm{\bu}) + 9 \sqrt{5 + \frac{1}{2\alpha -1}} \right] \: .
\end{align*}
\end{proof}

\pagebreak
\section{Proofs for Section~\ref{section:proof-convex-last-iterate}}
\label{section:last-iterate}

%In this section we prove Theorem~\ref{theorem:convex-last-iterate}. The proof
%relies on several lemmas.

%
\begin{proof}[Proof of Lemma~\ref{lemma:last-average}]
Let $S_k=\frac{1}{k} \sum_{t=T-k+1}^T \eta_t q_t$. We have
\[
\sum_{t=T-k}^T \eta_t (q_t - q_{T-k})
\ge \sum_{t=T-k}^T (\eta_t q_t - \eta_{T-k} q_{T-k})
= (k+1) S_{k+1} - \eta_{T-k} (k+1) q_{T-k} \: .
\]
That implies
\[
S_{k+1} - \eta_{T-k} q_{T-k} \le \frac{1}{k+1} \sum_{t=T-k}^T \eta_t (q_t - q_{T-k}) \: .
\]
From the definition of $S_k$ and the above inequality, we have
\[
k S_k
= (k+1) S_{k+1} - \eta_{T-k} q_{T-k}
= k S_{k+1} + S_{k+1} - \eta_{T-k} q_{T-k}
\le k S_{k+1} + \frac{1}{k+1} \sum_{t=T-k}^T \eta_t (q_t - q_{T-k}) \: .
\]
Therefore,
\[
S_{k} \le S_{k+1} + \frac{1}{k(k+1)} \sum_{t=T-k}^T \eta_t (q_t - q_{T-k}) \: .
\]
Unrolling the inequality we get
\[
\eta_T q_T = S_1 \le S_T + \sum_{k=1}^{T-1} \frac{1}{k(k+1)} \sum_{t=T-k}^T \eta_t (q_t - q_{T-k}) \: .
\]
Using the definition of $S_T$ the lemma follows.
\end{proof}

\begin{proof}[Proof of Lemma~\ref{lemma:difference-of-regularizers}]
Let $C_t=\frac{1}{2} \exp(\frac{1}{4} S_{t-1}^2 - Q_{t-1})$.
We will first show that
\begin{equation}
\label{equation:difference-of-regularizer-proof}
\phi_{T+1} (\bx_A) - \phi_A(\bx_A)
\le \max \left\{ C_{T+1}, \norm{\bx_A} \right\} (S_T^2 - S_{A-1}^2) \: .
\end{equation}

Now, we claim that $C_1, C_2, \dots$ is a non-decreasing sequence. Indeed,
$C_t \le C_{t+1}$ is equivalent to
\[
\frac{1}{4} S_{t-1}^2 - Q_{t-1} \le \frac{1}{4} S_t^2 - Q_t \: ,
\]
which is the same as
\[
\frac{1}{4} \norm{\bell_t}^2 - \frac{\norm{\bell_t}^2}{S_t^2} \ge 0 \: ,
\]
which trivially holds since $S_t^2 \ge 4$.

Second, observe that
\begin{equation}
\label{equation:Q-difference}
Q_T - Q_{A-1}
= \sum_{t=A}^{T} \frac{\norm{\bell_t}^2}{S_t^2}
\le \frac{1}{4} \sum_{t=A}^{T} \norm{\bell_t}^2
= \frac{1}{4} (S_T^2 - S_{A-1}^2) \: .
\end{equation}

We now prove \eqref{equation:difference-of-regularizer-proof} by considering three cases.

\textbf{Case $\norm{\bx_A} \le C_A$:}
\begin{align*}
\phi_{T+1} (\bx_A) - \phi_A(\bx_A)
& = \psi(\norm{\bx_A}, S_T, Q_T) - \psi(\norm{\bx_A},S_{A-1}, Q_{A-1}) \\
& \le \psi(C_A,S_T, Q_T) - \psi(C_A, S_{A-1}, Q_{A-1}) \\
& = \psi(C_A,S_T, Q_T) - C_A (S^2_{A-1}-2) \\
& = C_A \frac{\sqrt{2 S_T^2} (W(2 \exp(2Q_T) S_T^2 C_A^2)-1)}{\sqrt{W(2 \exp(2Q_T) S_T^2 C_A^2)}} - C_A (S^2_{A-1}-2) \\
& \le C_A \frac{\sqrt{2 S_T^2} (W(2 \exp(2Q_T) S_T^2 C_{T+1}^2)-1)}{\sqrt{W(2 \exp(2Q_T) S_T^2 C_{T+1}^2)}} - C_A (S^2_{A-1}-2) \\
& = C_A \frac{\sqrt{2 S_T^2} (\frac{1}{2} S^2_T-1)}{\sqrt{\frac12 S^2_T}} - C_A (S^2_{A-1}-2) \\
& = C_A (S^2_T-2) - C_A (S^2_{A-1}-2) \\
& = C_A (S^2_T - S^2_{A-1}) \\
& \le C_{T+1} (S^2_T - S^2_{A-1}) \: .
\end{align*}

\textbf{Case $C_A \le \norm{\bx_A} \le C_{T+1}$:}
\begin{align*}
\phi_{T+1} (\bx_A) - \phi_A(\bx_A)
& = \psi(\norm{\bx_A},S_{T}, Q_T) - \psi(\norm{\bx_A},S_{A-1}, Q_{A-1}) \\
& \le \psi(C_{T+1},S_{T}, Q_T) - \psi(C_{T+1},S_{A-1}, Q_{A-1}) \\
& = C_{T+1} \left(\frac{1}{2} S_T^2 + Q_T - \frac{1}{2} S_{A-1}^2 - Q_{A-1} \right)  \\
& \le C_{T+1} (S_T^2 - S_{A-1}^2) & \text{(using \eqref{equation:Q-difference})} \: .
\end{align*}

\textbf{Case 3: $\norm{\bx_A} \ge C_{T+1}$:}
\begin{align*}
\phi_{T+1} (\bx_A) - \phi_A(\bx_A)
& = \psi(\norm{\bx_A},S_{T}, Q_T) - \psi(\norm{\bx_A},S_{A-1}, Q_{A-1}) \\
& = \norm{\bx_A} \left(\frac{1}{2} S_{T}^2 + Q_{T} - \frac{1}{2} S^2_{A-1} - Q_{A-1}\right) \\
& \le \norm{\bx_A} (S_{T}^2 - S^2_{A-1}) & \text{(using \eqref{equation:Q-difference})} \: .
\end{align*}

Using the fact that $S_{T}^2 - S^2_{A-1} = \sum_{t=A}^T \|\bell_t\|^2$, \eqref{equation:difference-of-regularizer-proof} gives
\begin{align*}
\phi_{T+1} (\bx_A) - \phi_A(\bx_A)
&\le \max \left\{ \frac{1}{2} \exp\left(\frac{1}{4} S_T^2 - Q_T \right), \norm{\bx_A} \right\} \sum_{t=A}^T \|\bell_t\|^2 \\
&\le \max \left\{ \frac{1}{2} \exp\left(\frac{1}{4} S_T^2 - Q_T \right), \norm{\bx_A} \right\} \sum_{t=A}^T t^{-2\alpha},
\end{align*}
where we used $\norm{\bell_t} = \eta_t \norm{\bg_t} \le t^{-\alpha}$.
We take expectation to both sides of the above equation and define
\begin{equation}
\label{equation:definition-K}
\kappa = \sup_{\substack{T=0,1,2,\dots\\A=1,2,\dots}} \Exp\left[ \max \left\{ \frac{1}{2} \exp\left(\frac{1}{4} S_T^2 - Q_T \right), \norm{\bx_A} \right\} \right] \: ,
\end{equation}
to obtain
\[
\Exp\left[\phi_{T+1} (\bx_A) - \phi_A(\bx_A)\right]
\le \kappa \sum_{t=A}^T t^{-2\alpha} \: .
\]
We upper bound $\kappa$ as follows
\begin{align*}
\kappa
& = \sup_{\substack{T \ge 0\\ A\ge 1}} \Exp\left[ \max \left\{ \frac{1}{2} \exp\left(\frac{1}{4} S_T^2 - Q_T \right), \norm{\bx_A} \right\} \right] \\
& \le \sup_{A \ge 1} \Exp\left[ \max \left\{ \frac{1}{2} \exp\left(\frac{1}{4} S \right), \norm{\bx_A} \right\} \right] \\
& \le \sup_{A \ge 1} \frac{1}{2} \exp\left(\frac{1}{4} S \right) + \Exp\left[ \norm{\bx_A} \right] \\
& \le \sup_{A \ge 1} \frac{1}{2} \exp\left(\frac{1}{4} S \right) + \Exp\left[ \max\left\{\norm{\bx^*} + \sqrt{(S_{A-1} + 2)B_{\phi_A}(\bx^*, \bx_A)}, \ 2 \norm{\bx^*}, \ 4(S_{A-1} + 2) B_{\phi_A}(\bx^*, \bx_A) \right\} \right] \\
& \le \sup_{A \ge 1} \frac{1}{2} \exp\left(\frac{1}{4} S \right) + \Exp\left[ 3\norm{\bx^*} + \sqrt{(S_{A-1} + 2)B_{\phi_A}(\bx^*, \bx_A)} + 4(S_{A-1} + 2) B_{\phi_A}(\bx^*, \bx_A) \right] \\
& \le \sup_{A \ge 1} \frac{1}{2} \exp\left(\frac{1}{4} S \right) + \Exp\left[ 3\norm{\bx^*} + \sqrt{(S + 2)B_{\phi_A}(\bx^*, \bx_A)} + 4(S + 2) B_{\phi_A}(\bx^*, \bx_A) \right] \\
& \le \sup_{A \ge 1} \frac{1}{2} \exp\left(\frac{1}{4} S \right) + \sqrt{\Exp\left[(S + 2)B_{\phi_A}(\bx^*, \bx_A)\right]} + 3\norm{\bx^*} + \Exp\left[ 4(S + 2) B_{\phi_A}(\bx^*, \bx_A) \right] \\
& \le \sup_{A \ge 1} \frac{1}{2} \exp\left(\frac{1}{4} S \right) + \sqrt{(S+2) \Exp[1 + \phi_A(\bx^*)]} + 3\norm{\bx^*} + 4(S + 2) \Exp[1 + \phi_A(\bx^*)] \\
& \le \frac{1}{2} \exp\left(\frac{1}{4} S \right) + \sqrt{(S+2) \Exp[1 + \phi_\infty(\bx^*)]} + 3\norm{\bx^*} + 4(S + 2) \Exp[1 + \phi_\infty(\bx^*)] \\
& \le \frac{1}{2} \exp\left(\frac{1}{4} S \right) + (S+2) \sqrt{\Exp[1 + \phi_\infty(\bx^*)]} + 3\norm{\bx^*} + 4(S + 2) \Exp[1 + \phi_\infty(\bx^*)] \\
& \le \frac{1}{2} \exp\left(\frac{1}{4} S \right) + (S+2) \Exp[2 + \phi_\infty(\bx^*)] + 3\norm{\bx^*} + 4(S + 2) \Exp[2 + \phi_\infty(\bx^*)] \\
& = \frac{1}{2} \exp\left(\frac{1}{4} S \right) + 3\norm{\bx^*} + 5(S + 2) \Exp[2 + \phi_\infty(\bx^*)],
\end{align*}
where we used inequality $Q_T \ge 0$,
Lemma~\ref{lemma:bound-S-Q-psi-non-adaptive} that states $S_t \le S$ for all $t
\ge 0$, Lemma~\ref{lemma:iterate-bound} to upper bound $\norm{\bx_A}$,
inequality $\Exp[\sqrt{\cdot}] \le \sqrt{\Exp[\cdot]}$,
Lemma~\ref{lemma:regret-bound} to upper bound $\Exp[B_{\phi_A}(\bx^*, \bx_A)]$,
and inequality $\sqrt{x} \le 1 + x$ valid for any $x \ge 0$.
\end{proof}

\begin{proof}[Proof of Lemma~\ref{lemma:bound-sum-k}]
We upper bound the sum $\sum_{k=1}^{T-1} \frac{1}{k(k+1)} \sum_{t=T-k}^T t^{-2\alpha}$ as follows.
\begin{align*}
\sum_{k=1}^{T-1} \frac{1}{k(k+1)} \sum_{t=T-k}^T t^{-2\alpha}
& = \sum_{k=1}^{T-1} \sum_{t=T-k}^T t^{-2\alpha} \frac{1}{k(k+1)} \\
& = \sum_{k=1}^{T-1} T^{-2\alpha} \frac{1}{k(k+1)} + \sum_{t=1}^{T-1} \sum_{k=T-t}^{T-1} t^{-2\alpha} \frac{1}{k(k+1)} \\
& = T^{-2\alpha} \sum_{k=1}^{T-1} \left( \frac{1}{k} - \frac{1}{k+1} \right) + \sum_{t=1}^{T-1} t^{-2\alpha} \sum_{k=T-t}^{T-1} \left( \frac{1}{k} - \frac{1}{k+1} \right)  \\
& = T^{-2\alpha} \left( 1 - \frac{1}{T} \right) + \sum_{t=1}^{T-1} t^{-2\alpha} \left( \frac{1}{T-t} - \frac{1}{T} \right)  \\
& = T^{-2\alpha} \left( 1 - \frac{1}{T} \right) + \sum_{t=1}^{T-1} t^{-2\alpha} \frac{t}{T(T-t)} \\
& \le T^{-2\alpha} + \sum_{t=1}^{T-1} \frac{t^{1-2\alpha}}{T(T-t)} \: .
\end{align*}
Since $\alpha > \frac{1}{2}$, the function $t \mapsto t^{1-2\alpha}$ is convex on the interval $(0,\infty)$.
We can upper bound $t^{1-2\alpha}$ on the interval $[1,T]$ with a linear function. That is,
\begin{align*}
t^{1-2\alpha} & \le \frac{T-t}{T-1} + \frac{t-1}{T-1} T^{1-2\alpha} & \text{for $t \in [1,T]$} \: .
\end{align*}
Hence,
\begin{align*}
\sum_{k=1}^{T-1} \frac{1}{k(k+1)} \sum_{t=T-k}^T t^{-2\alpha}
& \le T^{-2\alpha} + \sum_{t=1}^{T-1} \frac{t^{1-2\alpha}}{T(T-t)}
 \le T^{-2\alpha} + \sum_{t=1}^{T-1} \left(\frac{1}{T (T-1)} + \frac{(t-1)T^{-2\alpha}}{(T-1)(T-t)}\right) \\
& = T^{-2\alpha} + \frac{1}{T} + \frac{T^{-2\alpha}}{(T-1)} \sum_{t=1}^{T-1} \frac{t-1}{T-t}
 = T^{-2\alpha} + \frac{1}{T} + \frac{T^{-2\alpha}}{(T-1)} \sum_{t=1}^{T-1} \left(\frac{T-1}{T-t} - 1 \right) \\
& = \frac{1}{T} + \frac{T^{-2\alpha}}{(T-1)} \sum_{t=1}^{T-1} \frac{T-1}{T-t}
 = \frac{1}{T} + T^{-2\alpha} \sum_{t=1}^{T-1} \frac{1}{T-t}
 = \frac{1}{T} + T^{-2\alpha} \sum_{t=1}^{T-1} \frac{1}{t} \\
& \le \frac{1}{T} + \frac{1 + \ln(T-1)}{T^{2\alpha}}
 \le \frac{1}{T} + T^{-2\alpha} + \frac{\ln T}{T^{2\alpha}}
 \le \frac{1}{T} + T^{-2\alpha} + \frac{1}{e(2\alpha - 1)T},
\end{align*}
where in the last step we used that $\ln x \le \frac{x^p}{e p}$ for all $x > 0$ and all $p > 0$ with $p=2\alpha-1$ and $x = T$.
\end{proof}

\begin{proof}[Proof of Lemma~\ref{lemma:partial-regret}]
For any $\bu \in \R^d$ and any $A \le T$,
\begin{align*}
- \sum_{t=A}^T \ip{\bell_t}{\bu}
& = \phi_{T+1}(\bu) - H_{T+1}(\bu) + \sum_{t=1}^{A-1} \langle \bell_t, \bu\rangle \\
& = \phi_{T+1}(\bu) - H_A(\bx_A) + H_A(\bx_A) - H_{T+1}(\bx_{T+1}) + H_{T+1}(\bx_{T+1}) - H_{T+1}(\bu) + \sum_{t=1}^{A-1} \ip{\bell_t}{\bu} \\
& = \phi_{T+1}(\bu) - H_A(\bx_A) + \sum_{t=A}^T \left[ H_t(\bx_t) - H_{t+1}(\bx_{t+1})\right] + H_{T+1}(\bx_{T+1}) - H_{T+1}(\bu) + \sum_{t=1}^{A-1} \ip{\bell_t}{\bu} \: .
\end{align*}
Adding $\sum_{t=A}^T \ip{\bell_t}{\bx_t}$ to both sides, we get
\begin{multline*}
\sum_{t=A}^T \ip{\bell_t}{\bx_t - \bu}
\le \phi_{T+1} (\bu) - H_A(\bx_A) + H_{T+1}(\bx_{T+1}) - H_{T+1}(\bu) \\ + \sum_{t=A}^T \left[ H_t(\bx_t) - H_{t+1}(\bx_{t+1}) + \ip{\bell_t}{\bx_t} \right] + \sum_{t=1}^{A-1} \ip{\bell_t}{\bu} \: .
\end{multline*}
Lemma~\ref{lemma:key-inequality} implies that
\begin{align*}
\sum_{t=A}^T \ip{\bell_t}{\bx_t - \bu} \le \phi_{T+1} (\bu) - H_A(\bx_A) + H_{T+1}(\bx_{T+1}) - H_{T+1}(\bu) + \sum_{t=1}^{A-1} \langle \bell_t, \bu\rangle \: .
\end{align*}
Set $\bu=\bx_A$, to obtain
\begin{align*}
\sum_{t=A}^T \ip{\bell_t}{\bx_t - \bx_A}
& \le \phi_{T+1} (\bx_A) - H_A(\bx_A) + H_{T+1}(\bx_{T+1}) - H_{T+1}(\bx_A) + \sum_{t=1}^{A-1} \langle \bell_t, \bx_A\rangle \\
& = \phi_{T+1} (\bx_A) - \phi_A(\bx_A) + H_{T+1}(\bx_{T+1}) - H_{T+1}(\bx_A) \\
& \le \phi_{T+1} (\bx_A) - \phi_A(\bx_A) \: .
\end{align*}
\end{proof}

\pagebreak
\section{Adaptive Learning Rate}
\label{section:adaptive-learning-rate}

\begin{lemma}[Bound on $S_\infty$ for adaptive learning rate]
\label{lemma:bound-S-adaptive}
Let $\alpha \in (\frac{1}{2}, 1)$. If $\eta_t = \frac{G^{2\alpha-1}}{\left(2G^2 + \sum_{i=1}^{t-1} \norm{\bg_i}^2 \right)^\alpha}$
then
$$
S_\infty \le \sqrt{4  + \frac{1}{2\alpha-1}} \: .
$$
\end{lemma}
\begin{proof}
For any $T \ge 1$,
\begin{align*}
S_T^2
& = 4 + \sum_{t=1}^T \norm{\bell_t}^2 \\
& = 4 + \sum_{t=1}^T \eta_t^2 \norm{\bg_t}^2 \\
& = 4 + G^{4\alpha-2} \sum_{t=1}^T \frac{\norm{\bg_t}^2}{(2G^2 + \sum_{i=1}^{t-1} \norm{\bg_i}^2)^{2\alpha}} \\
& \le 4 + G^{4 \alpha-2}  \sum_{t=1}^T \frac{\norm{\bg_t}^2}{(G^2 + \sum_{i=1}^{t} \norm{\bg_i}^2)^{2\alpha}} & \text{(since $\norm{\bg_t} \le G$)} \\
& \le 4 + G^{4 \alpha-2} \int_{G^2}^{G^2 + \sum_{t=1}^T \norm{\bg_t}^2} x^{-2 \alpha} dx  & \text{(Lemma~\ref{lemma:useful-inequality} with $f(x)=x^{-2\alpha}$, $a_0 = G^2$, $a_t = \norm{\bg_t}^2$)} \\
& = 4 + G^{4 \alpha-2} \frac{G^{2 - 4\alpha} - (G^2 + \sum_{t=j}^T \norm{\bg_t}^2)^{1-2\alpha}}{2 \alpha-1} \\
& \le 4 + \frac{1}{2 \alpha-1} \: .
\end{align*}
The lemma follows by taking limit $T \to \infty$.
\end{proof}

\begin{theorem}[Convergence rate for adaptive learning rate]
\label{theorem:convex-adaptive}
Let $F:\R^d \to \R$ be an $L$-smooth convex function with a minimizer $\bx^*$.
Suppose the stochastic gradients satisfy \eqref{equation:stochastic-gradient},
\eqref{equation:gradient-bound} and  $\Exp[\norm{\nabla F(\bx_t) - \bg_t}^2 ~|~ \mathcal{F}_t]\leq \sigma^2$.
Algorithm~\ref{algorithm:ftrl-rescaled-gradients-exponetial-potential} with learning rate
\[
\eta_t = \frac{G^{2\alpha-1}}{\left(2G^2 + \sum_{i=1}^{t-1} \norm{\bg_i}^2 \right)^\alpha} \: ,
\]
where $\alpha \in (\frac{1}{2}, 1)$ satisfies for all $T \ge 1$,
\begin{align*}
\Exp\left[ (F(\overline{\bx}_T) -F(\bx^*))^{1-\alpha}\right]
\le \frac{1}{T^{1-\alpha}} \max & \left\{ 2^\alpha G^{(1-2\alpha)(1-\alpha)}\left(1 + \phi_{\infty}(\bx^*-\bx_0) \right)^{1-\alpha} (2 G^2 + 2(T-1)\sigma^2)^{\alpha (1-\alpha)},\right.\\
&\quad \left.G^{1-2\alpha} 2^\frac{\alpha}{1-\alpha}\left(1 + \phi_{\infty}(\bx^*-\bx_0) \right) (4L)^{\alpha}\right\}\: .
\end{align*}
where
\[
\overline{\bx}_T = \frac{1}{T} \sum_{t=1}^T \bx_t \: .
\]
\end{theorem}
\begin{proof}
Lemma~\ref{lemma:regret-bound} states that for any $\bu \in \R^d$,
\[
\sum_{t=1}^T \Exp\left[ \eta_t \ip{\grad F(\bx_t)}{\bx_t - \bu} \right] \le \Exp\left[1 + \phi_{\infty}(\bu) \right] \: .
\]
Since $F$ is convex, $F(\bx_t) - F(\bu) \le \ip{\grad F(\bx_t)}{\bx_t - \bu}$ and therefore
\[
\sum_{t=1}^T  \Exp\left[ \eta_t (F(\bx_t) - F(\bu)) \right] \le \Exp\left[1 + \phi_{\infty}(\bu) \right] \: .
\]
Substituting $\bx^*$ for $\bu$, we have
\[
\sum_{t=1}^T  \Exp\left[ \eta_t (F(\bx_t) - F(\bx^*)) \right] \le \Exp\left[1 + \phi_{\infty}(\bx^*) \right] \: .
\]
Now observe that H\"{o}lder's inequality implies that $\Exp[B^p] \ge
\frac{\Exp[AB]^p}{\Exp[A^q]^{p/q}}$ for all $A,B$ non-negative random variables,
$p,q \ge 1$, and $\frac{1}{p} + \frac{1}{q} = 1$. Using it with $B=\left(\eta_T
\left(\sum_{t=1}^T (F(\bx_t) - F(\bx^*))\right)\right)^{1-\alpha}$ and
$A=\eta_T^{\alpha-1}$ and using the fact that the learning rates are decreasing
and $F(\bx_t) - F(\bx^*)$ are non-negative, we have
\[
\Exp\left[ 1 + \phi_{\infty}(\bx^*) \right]
\ge \Exp\left[\eta_T \sum_{t=1}^T  \left(F(\bx_t) - F(\bx^*)\right) \right]
\ge \frac{\Exp\left[\left(\sum_{t=1}^T  \left(F(\bx_t) - F(\bx^*)\right) \right)^{1-\alpha} \right]^\frac{1}{1-\alpha} }{\Exp\left[ \left(\frac{1}{\eta_T}\right)^\frac{1-\alpha}{\alpha}\right]^\frac{\alpha}{1-\alpha}}
 \: .
\]
Now, observe that
\begin{align*}
\Exp\left[ \left(\frac{1}{\eta_T}\right)^\frac{1-\alpha}{\alpha}\right]
& = G^{(1-2\alpha)(1-\alpha)/\alpha}\Exp\left[\left(2G^2+\sum_{t=1}^{T-1} \norm{\bg_t}^2 \right)^{1-\alpha}\right] \\
& \le G^{(1-2\alpha)(1-\alpha)/\alpha}\Exp\left[\left(2G^2+2\sum_{t=1}^{T-1} (\norm{\nabla F(\bx_t)-\bg_t}^2 + \norm{\nabla F(\bx_t)}^2)\right)^{1-\alpha}\right] \\
& \le G^{(1-2\alpha)(1-\alpha)/\alpha}\left(2G^2+2(T-1) \sigma^2\right)^{1-\alpha}+ G^{(1-2\alpha)(1-\alpha)/\alpha}\Exp\left[\left(2\sum_{t=1}^{T-1}\|\nabla F(\bx_t)\|^2\right)^{1-\alpha}\right] \\
& \le G^{(1-2\alpha)(1-\alpha)/\alpha}\left(2G^2+2(T-1) \sigma^2\right)^{1-\alpha}+ G^{(1-2\alpha)(1-\alpha)/\alpha}\Exp\left[\left(2\sum_{t=1}^{T}\|\nabla F(\bx_t)\|^2\right)^{1-\alpha}\right] \\
& \le G^{(2\alpha-1)(1-\alpha)/\alpha}\left(2G^2+2(T-1) \sigma^2\right)^{1-\alpha}+ G^{(1-2\alpha)(1-\alpha)/\alpha}\Exp\left[\left(4L\sum_{t=1}^{T}(F(\bx_t)-F(\bx^*))\right)^{1-\alpha}\right] \: .
\end{align*}
Putting all together and denoting by $\Delta=\sum_{t=1}^{T}(F(\bx_t)-F(\bx^*))$, we have
\[
\Exp\left[\Delta^{1-\alpha}\right]^\frac{1}{\alpha}
\le G^{(1-2\alpha)(1-\alpha)/\alpha}\left( \Exp \left[1 + \phi_{\infty}(\bu) \right] \right)^\frac{1-\alpha}{\alpha} \left((2G^2+2(T-1) \sigma^2)^{1-\alpha}+ (4L)^{1-\alpha}\Exp\left[\Delta^{1-\alpha}\right]\right)\: .
\]
With a case analysis, we have
\begin{align*}
\Exp\left[\Delta^{1-\alpha}\right]
\le \max&\left( 2^\alpha G^{(1-2\alpha)(1-\alpha)} \left(\Exp\left[1 + \phi_{\infty}(\bu) \right] \right)^{1-\alpha} (2 G^2 + 2(T-1)\sigma^2)^{\alpha (1-\alpha)},\right.\\
&\quad \left.G^{1-2\alpha} 2^\frac{\alpha}{1-\alpha} \Exp\left[1 + \phi_{\infty}(\bu) \right] (4L)^{\alpha}\right)\: .
\end{align*}
Jensen's inequality implies that $F(\overline{\bx}_T) \le \frac{1}{T} \sum_{t=1}^T  F(\bx_t)$, that gives the final bound.
\end{proof}

\end{document}